\documentclass{amsart}

\usepackage{amssymb}

\newtheorem{theorem}{Theorem}[section]
\newtheorem{lemma}[theorem]{Lemma}
\newtheorem{corollary}[theorem]{Corollary}
\newtheorem{proposition}[theorem]{Proposition}

\usepackage{epsfig}
\usepackage{enumerate}
\usepackage{hyperref}

\newcommand{\abs}[1]{\lvert#1\rvert}

\newcommand{\ord}[1]{\mathrm{ord}_{#1}}
\newcommand{\Q}{\mathbb{Q}}
\newcommand{\rad}[1]{\mathrm{rad}_{#1}}
\newcommand{\Z}{\mathbb{Z}}

\def\mod#1{{\ifmmode\text{\rm\ (mod~$#1$)}
\else\discretionary{}{}{\hbox{ }}\rm(mod~$#1$)\fi}}

\begin{document}
\title{Sums of two $S$-units via Frey-Hellegouarch curves}

\author{Michael A. Bennett}
\address{Department of Mathematics, University of British Columbia, Vancouver B.C., Canada}
\email{bennett@math.ubc.edu}

\author{Nicolas Billerey}
\address{(1) Universit\'e Clermont Auvergne, Universit\'e Blaise Pascal,
	Laboratoire de Math\'ematiques,
	BP 10448,
	F-63000 Clermont-Ferrand, France. 
	(2) CNRS, UMR 6620, LM, F-63171 Aubi\`ere, France}
\email{Nicolas.Billerey@math.univ-bpclermont.fr}
\thanks{The first named author was supported in part by a grant from NSERC. The second named author acknowledges the financial support of CNRS and ANR-14-CE25-0015 Gardio. He also warmly thanks the PIMS and the Mathematics Department of UBC for hospitality and excellent working conditions.}

\subjclass{Primary 11D61, Secondary 11G05}

\date{\today}
\keywords{}

\begin{abstract}
In this paper, we develop a new method for finding all perfect powers which can be expressed as the sum of two  rational $S$-units, where $S$ is a finite set of primes. Our approach is based upon the modularity of Galois representations and, for the most part, does not require lower bounds for linear forms in logarithms. Its main virtue is that it enables to carry out such a program explicitly, at least for certain small sets of primes $S$; we do so for $S = \{ 2, 3 \}$ and $S= \{ 3, 5, 7 \}$.
\end{abstract}

\maketitle

\section{Introduction}

If $S=\{ p_1, p_2, \ldots, p_k \}$ is a finite set of primes, we define the set of {\it $S$-units} to be those integers
of the shape $\pm p_1^{\alpha_1} p_2^{\alpha_2} \cdots p_k^{\alpha_k}$, with exponents $\alpha_i$ nonnegative integers. The arithmetic of such sets has been frequently studied due to its connections to a wide variety of problems in Number Theory and Arithmetic Geometry. In the latter direction, equations of the shape 
\begin{equation} \label{two}
x+y=z^2,
\end{equation}
where $x$ and $y$ are $S$-units for certain specific sets $S$, arise naturally when one wishes to make effective  a theorem of Shafarevich on the finiteness of isomorphism classes of elliptic curves over a number field $K$ with good reduction outside a given finite set of primes. By way of a simple example, if we wish to find all elliptic curves $E/\mathbb{Q}$ with nontrivial rational $2$-torsion and good reduction outside $\{ p_1, p_2, \ldots, p_k \}$, we are led to consider curves $E$ of the shape
$$
E \; : \; y^2 = x^3 + ax^2 + bx,
$$
where $a$ and $b$ are rational integers satisfying
$$
b^2 (a^2-4b) = \pm 2^{\alpha_0} p_1^{\alpha_1} \cdots p_k^{\alpha_k},
$$
for nonnegative integers $\alpha_i$. Writing $|b| = 2^{\beta_0} p_1^{\beta_1} \cdots p_k^{\beta_k}$,
we thus seek to solve equation (\ref{two}), with $z=a$ and $S = \{ 2, p_1, p_2, \ldots, p_k \}$.

An algorithm for computing all solutions to equations of the shape (\ref{two}), over $\Q$, can be found in Chapter 7 of de Weger \cite{Weg}, where, for instance, one can find a complete characterization of the solutions to equation (\ref{two}) in case $S= \{ 2, 3, 5, 7 \}$. This algorithm combines lower bounds for linear forms in complex and $p$-adic logarithms with lattice basis reduction for $p$-adic lattices.

More generally, for a given set of primes \(S\), we may consider equations of the shape
\begin{equation} \label{enn}
x+y=z^n,
\end{equation}
with $n \geq 2$ an integer, $x$ and $y$ \(S\)-units, and \(z\) a nonzero integer. We will call such a quadruple~\((x,y,z,n)\)  a {\it primitive} solution of (\ref{enn}) if \(\gcd(x,y)\) is \(n\)th-power free. Such equations are the main topic of discussion in Chapter 9 of Shorey and Tijdeman \cite{ShTi86}, due to their connections to the problem of characterizing perfect powers in nondegenerate binary recurrence sequences of algebraic numbers.
If we write $y=y_0 y_1^n$ in equation (\ref{enn}), with $y_0$ $n$th-power free (whereby there are at most $2 n^{|S|}$ choices for $y_0$), then it follows from the theory  of Thue-Mahler equations that, for a fixed value of \(n\ge3\), the set of primitive solutions \((x,y,z,n)\) of (\ref{enn}) is finite (see \cite[Thm.~7.2]{ShTi86}).  A stronger statement still is the following (essentially  Theorem 9.2 of \cite{ShTi86}) :

\begin{theorem} \label{ShTi-thm9.2}
There are only finitely many coprime \(S\)-units $x, y$ and $w$ for which there exist integers~\(n\ge2\) and~\(z\not=0\) such that~\(x+y=wz^n\). 
\end{theorem}

The aim of this paper is to illustrate the use of Frey-Hellegouarch curves in solving equations like (\ref{two}) and, more generally, ~(\ref{enn}). Specifically, we will apply such an approach to provide a new proof of Theorem \ref{ShTi-thm9.2} that {\it a priori} avoids the use of lower bounds for linear forms in logarithms, instead combining Frey-Hellegouarch curves, modularity and level-lowering, with the aforementioned theorem of Shafarevich. In fairness, it must be mentioned, that effective versions of the latter result have typically depended fundamentally on linear forms in logarithms; for recent papers along these lines, see  the work of Fuchs, von K\"anel and W\"ustholz \cite{FVW} and von K\"anel \cite{Kan}. The benefit of our approach is it enables us to, in Section \ref{s:all_solutions}, explicitly solve equation (\ref{enn}) for a pair of sets $S$ with cardinality $|S| \geq 2$. To the best of our knowledge, this is the first time this has been carried out. Indeed, it is unclear whether the classical approach to Theorem \ref{ShTi-thm9.2} via only lower bounds for linear forms in logarithms can be made practical with current technology, in any nontrivial situations.

The outline of our paper is as follows. Section \ref{sec2} introduces our basic notation.
In Section \ref{s:finiteness}, we show how to obtain various finiteness results currently proved with techniques from Diophantine approximation, via Frey-Hellegouarch curves over $\mathbb{Q}$. Philosophically, this bears a strong resemblance to recent work of von K\"anel \cite{Kan} and of Murty and Pasten \cite{MP}.
Section \ref{sec4} contains explicit details of the connections between Frey-Hellegouarch curves and modular forms. In Sections \ref{s:n=2}
 and~\ref{s:n=3}, we carry out such a ``modular'' approach quite explicitly for exponents~\(n=2\) and~\(n=3\) respectively. 
As an illustration of our methods, we completely solve~(\ref{enn}) for~\(S=\{2,3,5,7\}\) and~ $n \in \{ 2, 3 \}$ (hence recovering de Weger's aforementioned result), and also for $S = \{ 2, 3 , p \}$ and $n \in \{ 2, 3 \}$, for every prime $p < 100$. It should be emphasized that this is not a ``serious'' application of our method, but merely meant as an illustration of a partial converse of the connection between solving equations of the shape (\ref{enn}) and computing elliptic curves. The reader may wish to omit these sections at first (and, for that matter, subsequent) readings. A more interesting result along these lines is due to Kim \cite{Kim}, where the connection between more general cubic Thue-Mahler equation and Shafarevich's theorem is mapped out. 

Section \ref{s:all_solutions} contains, as previously mentioned,  the main result of the paper, namely an explicit solution of equation (\ref{enn}) for the sets $S = \{ 2, 3 \}$ and $\{ 3, 5, 7 \}$. 
The techniques we employ to prove these results, besides the aforementioned use of Frey-Hellegouarch curves and their associated modular forms, are local methods and appeal to computer algebra packages for solving Thue and Thue-Mahler equations; for the last of these, we rely extensively upon the computational number theory packages MAGMA \cite{magma}, PARI \cite{PARI2} and SAGE \cite{sage}.

We thank Rafael von K{\"a}nel and Benjamin Matschke for pointing out to us a missing solution in a previous version of Proposition~\ref{prop:local_obstructions_2-3-p-n=2}.

\section{Notation} \label{sec2}

In what follows, we will let~\(p\) be a prime number and~\(m\) a nonzero integer. We denote by~\(\rad{}(m)\) the radical of~\(|m|\), i.e. the product of distinct primes dividing~$m$, and by $\mbox{ord}_p(m)$ the largest nonnegative integer $k$ such that $p^k$ divides~$m$. We write \(\rad{p}(m)\) for the prime-to-\(p\) part of the radical of \(|m|\), that is, the largest divisor of $\rad{}(m)$ that is relatively prime to $p$.

We will begin by noting some basic results on modular  forms and connections between them and elliptic curves. Suppose that the $q$-expansion
\begin{equation} \label{newform}
f=q + \sum_{i \geq 2} c_i q^i
\end{equation}
defines a weight $2$, level $N_0$ (cuspidal) newform, with coefficients $c_i$ generating a number field $K/\mathbb{Q}$.
Further, let $E/\mathbb{Q}$ be an elliptic curve of conductor $N$ and $n$ a rational prime. If $l$ is a prime satisfying $l \nmid N$, we define
$$
a_l (E) = l+1 - \# E(\mathbb{F}_l).
$$
We say that $E$ {\it arises modulo $n$ from the newform $f$} and write $E \sim_n f$ if there exists a prime ideal $\mathfrak{N} \mid n$ of $K$ such that, given any prime $l \neq n$, we have either
\begin{equation} \label{tech1}
\mbox{$a_l (E) \equiv c_l \mod{\mathfrak{N}}$, if $l \nmid n N N_0$ }
\end{equation}
or
\begin{equation} \label{tech2}
\mbox{$l+1 \equiv \pm c_l \mod{\mathfrak{N}}$,  if $l \nmid n N_0$ and $\mbox{ord}_l(N)=1$. }
\end{equation}

\section{Finiteness results via modularity and level-lowering}\label{s:finiteness}

Throughout this section, we will let \(S\) denote a finite set of primes and~\(a\) a positive integer. The second part of the following lemma is a classical and easy application of the theory of linear forms in logarithms (see Corollary 1.2 of ~\cite{ShTi86}). We here give a complete proof of the result below using instead Frey-Hellegouarch curves and Shafarevich's theorem (Theorem IX.6.1 of \cite{Sil09}).
\begin{lemma}\label{lem:sum_three_S_units}
There are only finitely many \(S\)-units \(x,y,w\) with~\(\gcd(x,y)\le a\) for which there exists a nonzero integer~\(z\) such that~\(x+y=wz^2\). 
In particular, if~\(x,y\) are \(S\)-units with~\(\gcd(x,y)\le a\) such that~\(x+y\) is again a~\(S\)-unit, then \(\max \{\abs{x},\abs{y} \}\) is bounded by a constant depending on~\(S\) and~\(a\). 
\end{lemma}
\begin{proof}
Let $x, y$ and $w$ be \(S\)-units and let~\(z\not=0\) be an integer such that~\(x+y=wz^2\). Consider the elliptic curve
\[
E \; : \;  Y^2=X^3+2wzX^2+ywX
\]
with discriminant
\[
\Delta(E)=2^6x^2y^2w^3.
\]
It follows that \(E\) has good reduction outside~\(S\cup\{2\}\). By Shafarevich's theorem, there are only finitely many isomorphism classes of rational elliptic curves having good reduction outside a finite set of primes. This forces
\[
j(E)=j(x,y)=2^6\cdot\frac{(4x+y)^3}{y^2x}
\]
to take only finitely many values when \(x\) and~\(y\) range over all~\(S\)-units. But then, since \(j(x,y)=j(x/y,1)\), the quotient~\(x/y\) also takes finitely many different values and therefore~\(\max \{ \abs{x},\abs{y} \}\) is bounded whenever~\(\gcd(x,y)\le a\).   
\end{proof}

We next prove a version of Theorem~9.1 of~\cite{ShTi86} through appeal to (various) Frey-Hellegouarch curves and level-lowering.
\begin{theorem}\label{thm:ST}
Let $x, y$ and $w$ be \(S\)-units with~\(\gcd(x,y)\le a\). Let~\(n\ge2\) and  \(\abs{z}>1\) be integers. If \(x+y=wz^n\), then~\(n\) is bounded by a constant depending only on~\(S\) and~\(a\).
\end{theorem}
\begin{proof}
Assume~\(z\not=\pm1\). If~\(z\) is an~\(S\)-unit, then, by the previous lemma, ~\(\max \{ \abs{x},\abs{y} \}\) and thus \(n\) is bounded by a constant depending only on~\(S\) and~\(a\). Therefore, one may assume that~\(z\) is not a \(S\)-unit and, in particular, we may choose a prime~\(q\not\in S\)  dividing $z$. Define an elliptic curve~\(E/\Q\) as follows. If~\(q\not=2\), then consider 
$$
E \; : \;  Y^2=X(X-x)(X+y)
$$
whose discriminant and~\(c_4\)-coefficient are given by 
\[
\Delta(E)=2^4(xywz^n)^2\quad\text{and}\quad c_4(E)=2^4(w^2z^{2n}-xy).
\]
If however \(q=2\), we take
$$
E \; \colon \; Y^2+3xXY-x^2yY=X^3
$$
with discriminant and~\(c_4\)-coefficient  given by 
\[
\Delta(E)=-3^3x^8y^3wz^n\quad\text{and}\quad c_4(E)=3^2x^3(9wz^n-y).
\]
In both cases, \(E\) has multiplicative reduction at~\(q\) with~\(\ord{q}(\Delta(E))\) divisible by~\(n\). Let us then further assume that $n \geq 7$, \(n\not\in S\) and that the mod~\(n\) representation attached to~\(E\) is absolutely irreducible. By classical bounds on conductors (see~\cite{BrKR94} for instance), modularity of  elliptic curves over $\mathbb{Q}$ (\cite{BCDT01}) and Ribet's level lowering theorem (\cite{Rib90}), the elliptic curve~\(E\) arises modulo~\(n\) from a weight $2$ newform \(f(z)=\sum_{m\ge1}c_me^{2i\pi m z}\)  of (trivial Nebentypus and) level~\(N_0\mid\displaystyle{2^8\cdot3^5\cdot\prod_{\substack{p\in S\\ p\not=2,3}}p^2}\) such that~\(N_0\) is  coprime to~\(q\). Therefore, \(N_0\) is bounded by a constant depending only on~\(S\) and there exists a prime ideal~\(\mathfrak{N}\) above~\(n\) in the ring of integers of the coefficient field \(K\) of~\(f\) such that  
\[
c_q\equiv \pm(q+1)\mod{\mathfrak{N}}.
\]
By Deligne's bounds, \(c_q\pm(q+1)\) is a nonzero algebraic integer in~\(K\) whose Galois conjugates are all less than \((1+\sqrt{q})^2\) in absolute value and whose norm is divisible by~\(n\). Therefore we have \(n\le (1+\sqrt{q})^{2[K:\Q]}\). Since \([K:\Q]\) is bounded by the dimension of the space of weight-\(2\) cuspforms of level \(N_0\), it follows  that~\(n\) is bounded from above by a constant depending only on~\(S\), as desired. 
\end{proof}

Combining these results with the finiteness of the number of solutions to Thue-Mahler equations (see, for instance, Theorem 7.1 of \cite{ShTi86}) allows us to prove the following.
\begin{theorem}\label{thm:finiteness}
There are only finitely many \(S\)-units~\(x,y,w\) with~\(\gcd(x,y)\le a\) for which there exist integers~\(n\ge2\) and~\(z\not=0\) such that~\(x+y=wz^n\). 
\end{theorem}
\begin{proof}
By Lemma~\ref{lem:sum_three_S_units}, the equation~\(x+y=\pm w\) has only finitely many solutions in~\(S\)-units $x, y$ and $w$, with~\(\gcd(x,y)\le a\). Further, if $x, y$ and $w$ are \(S\)-units with \(\gcd(x,y)\le a\) such that there exist integers~\(n\ge2\) and~\(z\) with~\(\abs{z}>1\) satisfying~\(x+y=wz^n\), then, by the previous theorem, \(n\) is bounded by a constant depending only on~\(S\) and~\(a\).  But, for a fixed value of~\(n\ge2\), the finiteness of solutions to equation~\(x+y=wz^n\) in \(S\)-units~\(x,y,w\) with~\(\gcd(x,y)\le a\) and integer~\(z\) follows from Lemma~\ref{lem:sum_three_S_units} and from the finiteness of the set of solutions to Thue-Mahler equations for~\(n=2\) and~\(n\ge3\) respectively. 
\end{proof}

Given a primitive solution~\((x,y,z,n)\) of~(\ref{enn}), let us denote by \(d\) the gcd of~\(x\) and~\(y\). Then \(d\) divides \(z^n\) and, without loss of generality, we may write \(z^n/d=wz'^n\) where \(z'\) is a nonzero integer and \(w\) is a \(n\)th-power free positive \(S\)-unit. If~\(x'=x/d\) and~\(y'=y/d\), then one has
\begin{equation}\label{enn_modified}
x'+y'=wz'^n.
\end{equation}
Conversely, let~\(x',y',w\) be pairwise coprime \(S\)-units with~\(w\) positive satisfying the above equation and let~\(w'\) be a positive \(S\)-unit such that~\(ww'=\rad{}(w)^n\). Put~\(x=w'x'\), \(y=w'y'\) and~\(z=\rad{}(w)z'\). Then, \((x,y,z,n)\) is a primitive solution of~(\ref{enn}). 

Therefore, all the primitive solutions of equation~(\ref{enn}) can be deduced from the finite set of~\(S\)-units satisfying the condition of Theorem~\ref{thm:finiteness} with~\(a=1\). Moreover, combining this remark with the previous results, we have the  following :
\begin{corollary}
For a fixed value of~\(n\ge2\), there are only finitely many triples \((x,y,z)\) such that~\(x+y=z^n\) with~\(z\) nonzero integer, \(x,y\) \(S\)-units and~\(\gcd(x,y)\) \(n\)th-power free. Moreover,
there are only finitely many primitive solutions to equation~(\ref{enn}) if and only if~\(2\not\in S\).
\end{corollary}
\begin{proof}
According to the discussion above, the first part of the corollary is a direct consequence of Theorem~\ref{thm:finiteness} (with~\(a=1\)). If, however,~\(2\in S\), then \((2^{n-1},2^{n-1},2,n)\) is a primitive solution to~(\ref{enn}) for any~\(n\ge2\). 

Conversely, if~\(2\not\in S\), consider a primitive solution~\((x,y,z,n)\) to~(\ref{enn}). Dividing the equation by $\gcd (x,y)$ leads, as explained earlier, to an equation of the shape~\(x'+y'=wz'^n\) where~\(x',y',w\) are coprime \(S\)-units and~\(z'\mid z\). By assumption, \(x',y',w\) are odd and therefore~\(z'\) is even. In particular, we have~\(\abs{z'}>1\) and by Theorem~\ref{thm:ST}, \(n\) is bounded independently of~\(x,y\) and~\(z\). The desired finiteness result now follows from the first part of the corollary.
\end{proof}

In the proof of Lemma~\ref{lem:sum_three_S_units} we have appealed to \((n,n,2)\)-Frey-Hellegouarch curves and, for Theorem~\ref{thm:ST}, to~\((n,n,n)\) and~\((n,n,3)\)-Frey-Hellegouarch curves. The main goal of this paper is to make the statements of this section completely explicit in a number of situations. For this purpose, we will have use of  refined information on Frey-Hellegouarch curves of the above signatures. 

\section{Background on Frey-Hellegouarch curves and modular forms} \label{sec4}

We recall some (by now)  classical but useful results on Frey-Hellegouarch curves associated with generalized Fermat equations of signature~\((n,n,n)\), \((n,n,2)\) and \((n,n,3)\), and the newforms from which they arise. 


\subsection{Signature $(n,n,n)$}\label{ss:nnn}
Let $A, B$ and $C$ be $n$th-power free pairwise coprime nonzero integers and let $a, b$ and $c$ be pairwise coprime nonzero integers such that
\begin{equation} \label{eq-nnn}
Aa^n+Bb^n=Cc^n.
\end{equation}
We make the additional simplifying assumptions (which are not without loss of generality, but will be satisfied in the cases of interest to us) that
$$
Aa^n\equiv-1\mod{4} \; \mbox{ and } \;  Bb^n\equiv0\mod{16}.
$$
Define an elliptic curve $E_{n,n,n}^{A,B,C}(a,b,c)$ via
$$
E_{n,n,n}^{A,B,C}(a,b,c) \; : \;  Y^2+XY=X^3+\frac{Bb^n-Aa^n-1}{4}X^2-\frac{AB(ab)^n}{16}X.
$$
We summarize the properties of~\(E_{n,n,n}^{A,B,C}(a,b,c)\) that will be useful to us in the following result (see Kraus \cite{Kra97}). 
\begin{proposition} \label{nnn}
	If $n \geq 5$ is prime and $n \nmid ABC$, we have that 
	$$
	E=E_{n,n,n}^{A,B,C}(a,b,c)\ \sim_n f,
	$$
	for $f$ a weight $2$ cuspidal newform  of level
	$$
	N_0=\left\{\begin{array}{ll}
	2\, \rad{2}(ABC) & \text{if \; \(0\le\ord{2}(B)\le3\) or~\(\ord{2}(B)\ge5\)} \\
	\rad{2}(ABC) & \text{if \; \(\ord{2}(B)=4.\)} \\
	\end{array}
	\right.
	$$
	Further, if $l$ is prime with $l \nmid ABCabc$, then 
	$$
	a_l(E) \equiv l+1 \mod{4}.
	$$
\end{proposition}


\subsection{Signature $(n,n,2)$}\label{ss:nn2}
Next let $a, b, c, A, B$ and $C$ be nonzero integers such that
\begin{equation} \label{eq-nn2}
Aa^n+Bb^n=Cc^2,
\end{equation}
with $aA, bB$ and $cC$ pairwise coprime, $C$ squarefree and \(n\geq7\) prime. Without loss of generality, we may suppose that \(A\) and \(B\) are \(n\)th-power free and that we are in one of the following situations~:

\begin{enumerate}
	\item\label{item:nn2_case_i}  $abABC \equiv 1\mod{2}$ and $b \equiv -BC\mod{4}$; \\
	\item\label{item:nn2_case_ii}  $ab \equiv 1\mod{2}$ and either $\ord{2} (C) = 1$ or $\ord{2} (B) = 1$;  \\
	\item\label{item:nn2_case_iii}  $ab \equiv 1\mod{2}$, $\ord{2} (B) = 2$ and $c \equiv -bB/4\mod{4}$; \\
	\item\label{item:nn2_case_iv} $ab \equiv 1\mod{2}$, $\ord{2} (B) \in \{ 3, 4, 5 \}$ and $c \equiv C\mod{4}$;  \\
	\item\label{item:nn2_case_v} $\ord{2} (Bb^n) \geq 6$ and $c \equiv C\mod{4}$. 
\end{enumerate}
In cases~(\ref{item:nn2_case_i}) and~(\ref{item:nn2_case_ii}), we will consider the curve
$$
E_{(1),n,n,2}^{A,B,C}(a,b,c) \; \colon \; Y^2 = X^3 + 2cC X^2 + BC b^n X.
$$
In cases~(\ref{item:nn2_case_iii}) and~(\ref{item:nn2_case_iv}), we will instead consider
$$
E_{(2),n,n,2}^{A,B,C}(a,b,c) \; \colon \; Y^2 = X^3 + cC X^2 + \frac{BC b^n}{4} X,
$$
and in case (\ref{item:nn2_case_v}),
$$
E_{(3),n,n,2}^{A,B,C}(a,b,c) \; \colon \; Y^2 + XY = X^3 + \frac{cC-1}{4} X^2 + \frac{BC b^n}{64} X.
$$
These are all elliptic curves defined over~$\Q$.

The following lemma summarizes some useful facts about these curves. Apart from its (easy-to-check) assertion~(\ref{item:Lemma_nn2_2}) and up to some slight differences of notation, this is Lemma 2.1 of \cite{BeSk04}.
\begin{lemma}\label{lem:nn2_arith}
	Let $i=1, 2$ or $3$ and $E = E_{(i),n,n,2}^{A,B,C}(a,b,c)$.
	\begin{enumerate}
		
		\item The discriminant $\Delta (E)$ of the curve $E$ is given by
		$$
		\Delta(E) = 2^{\delta_i} C^3 B^2 A (ab^2)^n,\quad\text{where }\delta_i = 
		\begin{cases}
		6 & \text{if $i =1$} \\
		0 & \text{if $i= 2$} \\
		-12 & \text{if $i =3$}
		\end{cases}.
		$$
		\item\label{item:Lemma_nn2_2} The \(j\)-invariant \(j(E)\) of the curve $E$ is given by
		\[
		j(E)=2^6\frac{(4Aa^n+Bb^n)^3}{Aa^n(Bb^n)^2}.
		\]
		\item The conductor $N(E)$ of the curve $E$ is given by
		$$
		N (E)= 2^{\alpha} \rad{2}(C)^2 \rad{2}(abAB),
		$$
		where 
		$$
		\alpha = 
		\begin{cases}
		5 & \text{if $i =1$, case (\ref{item:nn2_case_i})} \\
		8 & \text{if $i= 1$, case (\ref{item:nn2_case_ii}) and $\ord{2} (C) =1$} \\
		7 & \text{if $i= 1$, case (\ref{item:nn2_case_ii}) and $\ord{2} (B) =1$} \\
		2 & \text{if $i=2$, case (\ref{item:nn2_case_iii}), $\ord{2} (B) =2$ and $b \equiv -BC/4\mod{4}$} \\
		3 & \text{if $i=2$, case (\ref{item:nn2_case_iii}), $\ord{2} (B) =2$ and $b \equiv BC/4\mod{4}$} \\
		5 & \text{if $i=2$, case (\ref{item:nn2_case_iv}) and $\ord{2} (B) = 3$} \\
		3 & \text{if $i=2$, case (\ref{item:nn2_case_iv}) and $\ord{2} (B) \in \{ 4, 5 \}$} \\
		0 & \text{if $i=3$, case (\ref{item:nn2_case_v}) and $\ord{2} (Bb^n) = 6$} \\
		1 & \text{if $i=3$, case (\ref{item:nn2_case_v}) and $\ord{2} (Bb^n) \geq 7$}.
		\end{cases}
		$$
		In particular, $E$ has  multiplicative reduction at each odd prime $p$ dividing $abAB$. Also, $E$ has  multiplicative reduction at $2$ if $\ord{2} (Bb^n) \geq 7$.
		\item The curve $E$ has a $\Q$-rational point of order $2$.
	\end{enumerate}
\end{lemma}

For the purposes of our applications, we will have need of an analog of Proposition \ref{nnn}, essentially Lemma 3.3 of \cite{BeSk04}.

\begin{proposition} \label{nn2}
	If $n \geq 7$ is prime and $ab\not=\pm1$, we have, for each $i \in \{ 1, 2, 3 \}$, that 
	$$
	E=E_{(i),n,n,2}^{A,B,C}(a,b,c)\ \sim_n f,
	$$
	for $f$ a weight $2$ cuspidal newform  of level $N_0=2^{\alpha'} \rad{2}(C)^2\rad{2}(AB)$ where
	$$
	\alpha'=\left\{
	\begin{array}{ll}
	1 & \text{if $ab\equiv0\mod{2}$ and $AB\equiv1\mod{2}$} \\
	\alpha & \text{otherwise},
	\end{array}
	\right.
	$$
	where \(\alpha\) is as defined in Lemma~\ref{lem:nn2_arith}.
\end{proposition}


\subsection{Signature $(n,n,3)$}\label{ss:nn3}
Let us suppose that $a, b, c, A, B$ and $C$ are nonzero integers such that
\begin{equation} \label{eq-nn3}
Aa^n+Bb^n=Cc^3
\end{equation}
with \(n\geq5\) prime. Assume $A a$, $B b$, and $C c$ are pairwise coprime, and, without loss of generality, that $A a \not \equiv 0 \mod{3}$ and $B b^n \not \equiv 2 \mod{3}$.
Further, suppose that $C$ is cubefree and that \(A\) and \(B\) are \(n\)th-power free. We consider the elliptic curve 
\[
E_{n,n,3}^{A,B,C}(a,b,c) \; \colon \;  Y^2+ 3 C c XY + C^2 B b^n Y=X^3.
\]
With the above assumptions, we have the following result (Lemma 2.1 of \cite{BeVaYa04}).
\begin{lemma}\label{lem:nn3_arith}
	Let $E=E_{n,n,3}^{A,B,C}(a,b,c)$. 
	\begin{enumerate}
		\item The discriminant $\Delta(E)$ of the curve $E$ is given
		by 
		\begin{equation*}
			\Delta(E)=3^3AB^3 C^8(ab^3)^n.
		\end{equation*}
		\item The \(j\)-invariant of~\(E\) is given by
		\begin{equation*}
			j(E)=3^3 \frac{C c^3(9 A a^n + B b^n)^3}{A B^3 (ab^3)^n}.
		\end{equation*}
		\item The conductor $N(E)$ of the curve $E$ is 
		$$
		N(E) = 3^{\alpha}\rad{3}(A B ab)\rad{3}(C)^2 
		$$
		where 
		\begin{eqnarray*}
			\alpha=
			\begin{cases}
				2 & \text{if $9 \mid (2+C^2 B b^n-3C c)$,} \\
				3 & \text{if $3 \parallel (2+C^2 B b^n - 3C c)$,} \\
				4 & \text{if $\ord{3}(B b^n)=1$}, \\
				3 & \text{if $\ord{3}(B b^n)=2$}, \\
				0 & \text{if $\ord{3}(B b^n)=3$}, \\
				1 & \text{if $\ord{3}(B b^n)>3$}, \\
				5 & \text{if $3 \mid C$}.
			\end{cases}
		\end{eqnarray*}  
		In particular, $E$ has split multiplicative reduction at each
		prime $p \neq 3$ dividing $B b$, split multiplicative
		reduction at each prime dividing $A a$ congruent to
		$1$, $4$, $5$, $7$, $16$, $17$ or $20$ modulo $21$, and non-split
		multiplicative reduction at all other primes dividing $A a$,
		except $3$. Also, $E$
		has split multiplicative reduction at $3$ if $\ord{3}(B b^n)>3$, and
		good reduction if $\ord{3}(B b^n)=3$.
		\item The curve $E$ has a $\Q$-rational point of order $3$.
	\end{enumerate}
	
\end{lemma}

In what follows, we will apply Frey-Hellegouarch curves of signature $(n,n,3)$ more frequently than other signatures; the following result, essentially Proposition 4.2 of \cite{BeVaYa04} (though it is worth noting that the value $N_n(E)$ as define in Lemma 3.4 of that paper is actually stated incorrectly for the cases where $n \mid ABC$), will be of particular use to us.

\begin{proposition}\label{lem:nn3_LL-1}
	If $n \geq 5$ is prime, $ab \neq \pm 1$ and $E_{n,n,3}^{A,B,C}(a,b,c)$ does not correspond to the identities
	$$
	1 \cdot 2^5 + 27 \cdot (-1)^5 = 5 \cdot 1^3 \; \mbox{ or } \;
	1 \cdot 2^7 + 3 \cdot (-1)^7 = 1 \cdot 5^3,
	$$
	then we have that 
	$$
	E=E_{n,n,3}^{A,B,C}(a,b,c)\ \sim_n f,
	$$
	for $f=\sum_{m\geq1}c_mq^m$ a weight $2$ cuspidal newform  of level
	\[
	N_0=3^{\alpha'}\rad{3}(AB)\rad{3}(C)^2,
	\]
	where~\(\alpha'=\alpha\) with~\(\alpha\) as defined in Lemma~\ref{lem:nn3_arith} unless~\(\ord{3}(Bb^n)\ge3\) in which case we have
	\[
	\alpha'=\left\{
	\begin{array}{ll}
	0 & \text{if $\ord{3}(B)=3$}, \\
	1 & \text{otherwise}.
	\end{array}
	\right.
	\]
	More precisely, if \(l\) is a prime, coprime to \(nN_0\), then \(n\) divides \(\mathrm{Norm}_{K/\mathbb{Q}}(c_l-a_l)\) where \(K\) is the number field generated by the Fourier coefficients of~\(f\) and \(a_l\in S_l\), with
	\[
	S_l=\{x\colon \abs{x}<2\sqrt{l},\ x\equiv l+1\mod{3}\}\cup\{l+1\},
	\]
	if \(l\equiv1,4,5,7,16,17,20\mod{21}\), and
	\[
	S_l=\{x\colon \abs{x}<2\sqrt{l},\ x\equiv l+1\mod{3}\}\cup\{-l-1,l+1\},
	\]
	otherwise.
\end{proposition}

\section{The case $n=2$}\label{s:n=2}

In this section, we  consider equation (\ref{two}) (as treated by de Weger in \cite{Weg} and \cite{Weg2}), via an \((n,n,2)\) Frey-Hellegouarch curve approach. According to the discussion of Section~\ref{s:finiteness}, the corresponding equations to treat are of the shape
\begin{equation} \label{gen}
x+y = w z^2
\end{equation}
where $x, y$ and $w$ are pairwise coprime $S$-units. Define \(a=b=1\), \(c=z\), \(A=x\), \(B=y\) and \(C=w\). Then we have~\(Aa^n+Bb^n=Cc^2\) and may assume, without loss of generality, that we are in one of the situations~(\ref{item:nn2_case_i})-(\ref{item:nn2_case_v}) of~\S\ref{ss:nn2}. Consider the associated elliptic curve~\(E/\Q\) of Lemma~\ref{lem:nn2_arith}. It has good reduction outside~\(S'=S\cup\{2\}\). Therefore, if we know representatives~\(F/\Q\) of all (the finitely many) isomorphism classes of rational elliptic curves (with a nontrivial two-torsion subgroup) having good reduction outside~\(S'\), all that remains to do to solve~(\ref{gen}) is to check for an equality
\[
j(E)=j(F),\quad\text{where \(j(E)=2^6\cdot\frac{(4x+y)^3}{y^2x}\)}
\]
and \(j(F)\) denote the \(j\)-invariants of~\(E\) and~\(F\) respectively. Computing such representatives is a classical but challenging problem that has only been achieved for a rather restrictive list of sets, including~\(\{2,3,5,7\}\), \(\{2,3,11\}\), \(\{2,13\}\), \(\{2,17\}\), \(\{2,19\}\) and~\(\{2,23\}\) (see \cite{CrLi07}).

Nevertheless, the precise information on the conductor~\(N(E)\) of~\(E\) provided by Lemma~\ref{lem:nn2_arith}, namely~\(N(E)=2^{\alpha} \rad{2}(w)^2 \rad{2}(xy)\) where
\[
\alpha = 
\begin{cases}
5 & \text{if $xyw \equiv 1\mod{2}$ and $yw \equiv -1\mod{4}$} \\
8 & \text{if $\ord{2}(w)=1$} \\
7 & \text{if $\ord{2}(y)=1$} \\
2 & \text{if $\ord{2}(y)=2$ and $z\equiv w\equiv -y/4 \mod{4}$} \\
3 & \text{if $\ord{2}(y)=2$ and $z\equiv -w\equiv -y/4 \mod{4}$} \\
5 & \text{if $\ord{2}(y)=3$ and $z\equiv w \mod{4}$} \\
3 & \text{if $\ord{2}(y)\in\{4,5\}$ and $z\equiv w$} \\
0 & \text{if $\ord{2}(y)=6$ and $z\equiv w \mod{4}$} \\
1 & \text{if $\ord{2}(y)\ge7$ and $z\equiv w \mod{4}$},
\end{cases}
\]
allows us to sometimes show, for some specific sets~\(S\), that we have~\(N(E)\le 350000\) except for some precisely identified quadruples~\((x,y,w,z)\). In that situation, we can therefore appeal, via the SAGE package \verb"database_cremona_ellcurve" (\cite{sage}, \cite{Cre}), to Cremona's tables which, at the time of writing (Spring 2015), contain representatives for all rational elliptic curves of conductor less than~\(350000\).

By way of example, consider~\(S=\{2,3,5,7\}\) or \(S=\{2,3,p\}\) with~\(p\) prime, \(11\le p<100\). We stress the fact that, for most of these latter sets, we presently do not know a complete list of representatives of the isomorphism classes of rational elliptic curves having good reduction outside~\(S\). Nevertheless, the strategy outlined above does apply and this allows us to easily recover de Weger's result mentioned  in the Introduction and to solve equations that remained unsolved in a recent paper  by Terai \cite{Ter14} (see Corollary~\ref{prop:Terai}). We carry out the details in the next two subsections;
the SAGE code used for our computations is available at

\hskip8ex \verb"http://www.math.ubc.ca/~bennett/Sum_Of_Two_S-units.pdf".

\subsection{The case $S=\{2,3,5,7\}$}

Specializing the approach above to the case \(S=\{2,3,5,7\}\) considered by de Weger, we note that {\it a priori} we must consider conductors of the shape
$$
N(E) = 2^\alpha\cdot 3^{\delta_3}\cdot 5^{\delta_5}\cdot 7^{\delta_7},
$$
where~$\delta_i=\ord{i}(N(E)) \le2$, \(i=3,5,7\), and $\alpha=\ord{2}(N(E)) \in \{ 0, 1, 2, 3, 5, 7, 8 \}$.  However, we have the following easy result.

\begin{lemma}\label{lem:local_obstructions}
In each case \(N(E)\le 350000\). Further, if \((\delta_3,\delta_5,\delta_7)\in\{0,2\}\), then either
\begin{multline*}
(x,y,w,z)\in\{(-1,8,7,-1),(-1,64,7,3),(1,4,5,-1),(-1,16,15,-1), \\
(1,2,3,\pm1),(-1,4,3,-1),(-1,2,1,\pm1),(1,8,1,-3),(1,1,2,\pm1)\},
\end{multline*}
or~\((\delta_3,\delta_5,\delta_7)\in\{(0,0,0),(0,0,2),(2,2,0)\}\) and~\(\alpha=1\).
\end{lemma}
\begin{proof}
If \(\delta_3,\delta_5,\delta_7\in\{0,2\}\),  we have a solution to an equation of the shape
\[
2^k+\epsilon=wz^2,
\]
where~\(k\) is nonnegative, \(\epsilon=\pm1\) and~\(\rad{}(w)\mid 2\cdot3\cdot5\cdot7\). The solutions to this equation for~\(k\le6\) correspond to those listed in the lemma. Moreover, if~\(k\ge7\), none of these equations has a solution modulo~\(840\) unless we have~\(w=1,7\) or~\(15\), that is~\((\delta_3,\delta_5,\delta_7)\in\{(0,0,0),(0,0,2),(2,2,0)\}\). 

Similarly, it is  easy to check that $N(E) \le 350000$, unless we have
$$
\alpha=8 \; \mbox{ and } \; (\delta_3,\delta_5,\delta_7) = (2,2,2), (2,2,1), (2,1,2), (1,2,2),
$$
$$
\alpha=7 \; \mbox{ and } \; (\delta_3,\delta_5,\delta_7) = (2,2,2),  (1,2,2), 
$$
or
$$
\alpha=5 \; \mbox{ and } \; (\delta_3,\delta_5,\delta_7) = (2,2,2). 
$$
Of these, only the cases
$$
(\alpha,\delta_3,\delta_5,\delta_7) =(8,2,2,1), (8,2,1,2), (8,1,2,2), (7,1,2,2)
$$
may correspond to solutions to (\ref{gen}), namely to equations of the shape
$$
7^k \pm 1 = 2 \cdot 3 \cdot 5 z^2, \; \; 5^k  \pm 1 = 2 \cdot 3 \cdot 7 z^2, \; \; 3^k  \pm 1 = 2 \cdot 5 \cdot 7 z^2, \; \;
 3^k \pm 2 = 5 \cdot 7 z^2,
$$
and $2 \cdot 3^k \pm 1 = 5 \cdot 7 z^2$. Each of these equations, however, is insoluble modulo~$840$.
\end{proof}
From Lemma \ref{lem:local_obstructions} and the above discussion, we may thus appeal to Cremona's SAGE package to reproduce de Weger's results (with a slightly faster search provided by the restrictions on the exponents given in the above lemma).
In particular, if \(S=\{2,3,5,7\}\), as noted in \cite[Thm.~7.2]{Weg}, there are precisely~\(388\) solutions~\((x,y,z)\) to~(\ref{two}) with \(\gcd(x,y)\) squarefree and \(x\ge\abs{y}>0\) and~\(z>0\). We list these solutions in the file
\verb"http://www.math.ubc.ca/~bennett/2-3-5-7.pdf".

As an obvious byproduct, we also obtain a complete list of solutions to~(\ref{two}) with~\(S=\{2,3\}\) and~\(S=\{3,5,7\}\). As we will consider these sets again in Section~\ref{s:all_solutions}, we state these results here.
\begin{proposition}
The only primitive solutions to equation~(\ref{two}) with \(S=\{2,3\}\) and with, say, $x \ge |y| > 0$ are given by
\[
\begin{array}{c}
(x,y) = (2,-1), (2,2), (3,-2), (3,1), (4,-3), (6,-2), (6,3), (8,1),  (9,-8), (12,-3), \\ 
 (16,9), (18,-2), (24,1), (27,-2), (48,1), (81,-32),  (288,1) \mbox{ and } (486,-2). \\
\end{array}
\]
\end{proposition}

\begin{proposition}
The only primitive solutions to equation~(\ref{two}) with \(S=\{3,5,7\}\) and with, say, $x \ge |y| > 0$ are given by
\[
\begin{array}{l}
(x,y) = (3,1), (5,-1), (7,-3), (9,-5), (9,7), (15,1), (21,-5), (21,15), (25,-21),\\
 (25,-9), (35,1), (49,-45), (49,15), (63,1), (105,-5), (135,-35), (147,-3), (175,21), \\
(175,81), (189,-125), (189,7), (343,-243), (405,-5), (625,-49), (675,1),\\
(729,-245), (1029,-5), (3375,2401), (3969,-125), (9375,1029), (15625,-1701),\\
(59535,1), (688905,-5) \mbox{ and } (4782969,4375). \\
\end{array}
\]
\end{proposition}

\subsection{The case $S=\{2,3,p\}$}
We now turn our attention to the case~\(S=\{2,3,p\}\) where~\(p\) is a prime in the range $11 \leq p < 100$. Once again, we must \emph{a priori} consider conductors of the shape
$$
N(E) = 2^\alpha\cdot 3^{\delta_3}\cdot p^{\delta_p},
$$
where~$\delta_i=\ord{i}(N(E)) \le2$, \(i=3,p\), and $\alpha=\ord{2}(N(E)) \in \{ 0, 1, 2, 3, 5, 7, 8 \}$; many of these conductors exceed the limits of the current Cremona database.  However, we have the following result.
\begin{proposition}\label{prop:local_obstructions_2-3-p-n=2}
In each case~\(N(E)\le 350000\), unless~\((x,y,w,z)\) corresponds to one of the following equations
\[
3^4+1=2\cdot41,\quad 2\cdot3^3-1=53,\quad 3^{10}-1=2\cdot61\cdot22^2,\quad  3^4+2=83,\quad 3^4-2=79,
\]
\[
3^5+1=61\cdot2^2,\quad 2^3\cdot3^2+1=73,\quad 2^3\cdot3^2-1=71,\quad 3^4+2^3=89,\quad 3^4-2^3=73.
\]
\end{proposition}
\begin{proof}
It is easy to check that we have~\(N(E)\le 350000\) unless we are in one of the following situations :
$$
\begin{array}{|c|c|c|c|c|c|} \hline
(\alpha,\delta_3,\delta_p) & (8,2,2) & (7,2,2) & (8,1,2) & (7,1,2) & (8,0,2) \\ \hline
p & \geq 13 & \geq 19 & \geq 23 & \geq 31 & \geq 37 \\ \hline \hline
(\alpha,\delta_3,\delta_p) & (5,2,2) & (7,0,2) & (5,1,2) & (3,2,2) &  \\ \hline
p & \geq 37 & \geq 53 & \geq 61 & \geq 71 &  \\ \hline
\end{array}
$$
Of these, only the cases~\((\alpha,\delta_3,\delta_p)\in\{(8,1,2),(7,1,2),(5,1,2)\}\) may actually correspond to our equations. We are therefore led to  solve the following equations, where~\(\epsilon=\pm1\) and~\(k\) is a positive integer :
\[
3^k+\epsilon=2pz^2,\quad\text{for \(p\ge23\)},
\]
\[
2\cdot3^k+\epsilon=pz^2\quad\text{and}\quad 3^k+2\epsilon=pz^2,\quad\text{for \(p\ge31\)}
\]
and
\[
3^k+\epsilon=pz^2,\quad 2^3\cdot 3^k+\epsilon=pz^2\quad\text{and}\quad 3^k+2^3\epsilon=pz^2,\quad\text{for \(p\ge61\)}.
\]
We  deal with each of these in turn. Considering first the equation~\(3^k+\epsilon=2pz^2\), we readily check that it has no solution modulo~\(8p\) for odd values of~\(k\) and $23 \leq p < 100$. Assume therefore that~\(k\) is even. If moreover~\(\epsilon=-1\), then by factorization, we end up with an equation of the shape $3^m\pm1=z'^2$ or $3^m\pm1=2z'^2$ for some integer~$z'$. According to~\cite[Thm.~1.1]{BeSk04}, this leads to a unique solution to our equation, namely $3^{10}-1=2\cdot61\cdot22^2$.

Finally, if~\(k\) is even and~\(\epsilon=+1\), the equation~\(3^k+1=2pz^2\) has no solution modulo~\(24p\) unless~\(p=29,41,53\) or~\(89\). However, for $p=29,53$ and $89$, it has no solution modulo~\(pq\) where $q=43,313$ and $23$ respectively. In the remaining case, namely~\(p=41\), write \(3^k=3^\beta x^3\) with~\(0\le \beta\le2\) and~\(x\in\Z\). Then, \((X,Y)=(2\cdot3^\beta \cdot 41 x,2^23^\beta \cdot 41^2 z)\) is an integral point on the elliptic curve
\[
Y^2=X^3+2^3 \cdot 3^{3\beta} \cdot 41^3.
\]
Computing its integral points using the aforementioned SAGE command thus leads to the unique solution~\(3^4+1=2\cdot41\).

We now turn our attention to the equation~\(2\cdot3^k+\epsilon=pz^2\). We easily check that for~\(p\) in the range~ $31 \leq p < 100$, this equation has no solution modulo~\(24p\) unless we have
\[
(\epsilon,p)=(1,31),\ (1,43),\ (1,79),\ (-1,53)\mbox{ or } (-1,89).
\]
For $p=31,43,79$ and $89$, the corresponding equation has no solution modulo~\(pq\) where $q=13,7,13$ and $23$ respectively. For the remaining case, that is \((\epsilon,p)=(-1,53)\), reducing modulo~\(53\), we find that~\(k\equiv0\mod{3}\). Writing \(k=3k_0\),   \((X,Y)=(2\cdot3^{k_0} \cdot 53,2 \cdot 53^2 z)\) is thus an integral point on the elliptic curve
\[
Y^2=X^3-2^2 \cdot 53^3.
\]
As before, we compute its integral points on SAGE and deduce the unique solution~\(2\cdot3^3-1=53\).

Consider now equation~\(3^k+2\epsilon=pz^2\) for \(p\) in the range~\(31\le p\le 100\). We check that there is no solution modulo~\(24p\) unless we have
\[
(\epsilon,p)=(1,53),\ (1,59),\ (1,83),\ (-1,31),\ (-1,79) \mbox{ and } (-1,97).
\]
For $p=53,59,31$ and $97$, we however have that the corresponding equation has no solution modulo~\(pq\) where $q=2887,523,13$ and $7$ respectively. It remains to deal with the cases~\((\epsilon,p)=(1,83)\) and~\((-1,79)\). In the latter case, we find that~\(k\equiv1\mod{3}\) by reducing modulo~\(79\). Writing~\(k=3k_0+1\), ~\((X,Y)=(3^{k_0+1} \cdot 79,3^{1} \cdot 79^2z)\) is necessarily an integral point on the elliptic curve
\[
Y^2=X^3-2\cdot3^2\cdot79^3.
\]
This again leads to a unique solution, namely~\(3^4-2=79\). In the former case, that is~\((\epsilon,p)=(1,83)\), write~\(3^k=3^\beta x^3\) with~\(0\le \beta\le2\) and~\(x\in\Z\). Then, \((X,Y)=(3^\beta \cdot 83 x,3^\beta \cdot 83^2 z)\) is an integral point on the elliptic curve
\[
Y^2=X^3+2\cdot3^{2\beta} \cdot 83^3.
\]
Computing its integral points using  the command \verb"IntegralPoints" in MAGMA (\cite{magma}) leads to a unique solution which is~\(3^4+2=83\). We note here that the corresponding routine in SAGE has marked difficulty with this equation.

We now consider equation~\(3^k+\epsilon=pz^2\). If~\(\epsilon=+1\), then we have no solution modulo~\(24p\) unless \(p=61,67\) or~\(79\). For~ $p=67$ or $79$,  however, we have that the corresponding equation has no solution modulo~\(pq\) where $q=23$ or $2341$ respectively. If \(p=61\), write \(3^k=3^\beta x^3\) with~\(0\le \beta\le2\) and~\(x\in\Z\). Then, \((X,Y)=(3^\beta \cdot 61 x,3^\beta \cdot 61^2 z)\) is an integral point on the elliptic curve
\[
Y^2=X^3+3^{2\beta}\cdot 61^3.
\]
Computing its integral points using SAGE leads to the unique solution~\(3^5+1=61\cdot2^2\). If now~\(\epsilon=-1\), then reducing~\(3^k-1=pz^2\) modulo~\(8\) shows that \(k\) is necessarily even. Write \(k=2k_0\) and \(3^{k_0}=3^\beta x^3\) with~\(0\le \beta\le2\) and~\(x\in\Z\). Then, for some divisor~\(z_1\) of~\(z\), either \((X,Y)=(2\cdot3^\beta x,2^2\cdot3^\beta z_1)\) or \((X,Y)=(2\cdot3^\beta x,2\cdot3^\beta z_1)\) is an integral point on one of the elliptic curves
\[
Y^2=X^3\pm2^3\cdot3^{2\beta}.
\]
However none of their integral points (which we have already computed) corresponds to a solution of our equation.

Let us now consider the equation~\(2^3\cdot3^k+\epsilon=pz^2\). If~\(\epsilon=+1\), then the corresponding equation has no solution modulo~\(24p\) unless~\(p=73\) or~\(97\). Moreover in the former case, we have~\(k\equiv2\mod{3}\). If~\(p=97\), we check that~\(2^3\cdot3^k+1=97z^2\) has no solution modulo~\(13\cdot97\). Assume thus that~\(p=73\) and write \(k=3k_0+2\). Then \((X,Y)=(2\cdot3^{2+k_0} \cdot 73,3^{2} \cdot 73^2 z)\) is an integral point on the elliptic curve
\[
Y^2=X^3+3^{4}\cdot 73^3.
\]
We therefore find that there is only one solution corresponding to~\(2^3\cdot3^2+1=73\). Similarly, if~\(\epsilon=-1\), then the corresponding equation has no solution modulo~\(24p\) unless~\(p=71\). Write \(3^k=3^\beta x^3\) with~\(x\in\Z\) and \(0\le\beta\le 2\). Then, \((X,Y)=(2\cdot3^\beta \cdot 71 x,3^\beta \cdot 71^2 z)\) is an integral point on the elliptic curve
\[
Y^2=X^3-3^{2\beta} \cdot 71^3.
\]
This gives rise to the unique solution~\(2^3\cdot3^2-1=71\).

We finally deal with the last equation, namely~\(3^k+2^3\epsilon=pz^2\). If~\(\epsilon=+1\), then the corresponding equation has no solution modulo~\(24p\) unless~\(p=83\) or~\(89\). If $p=83$, we find a local obstruction modulo $2^3 \cdot 7 \cdot 13$. If $p=89$, we write \(3^k=3^\beta x^3\) with~\(x\in\Z\) and \(0\le\beta\le2\), whereby \((X,Y)=(3^\beta \cdot 89 x,3^\beta \cdot 89^2 z)\) is an integral point on the elliptic curve
\[
Y^2=X^3+2^3\cdot3^{2\beta} \cdot 89^3.
\]
We therefore conclude that there is only one solution corresponding to~\(3^4+2^3=89\). Similarly, if~\(\epsilon=-1\), then the corresponding equation has no solution modulo~\(24p\) unless~\(p=67\) or~\(p=73\). Moreover, in that latter case, we have~\(k\equiv1\mod{3}\) and it is easy to check that~\(3^k-2^3=67z^2\) has no solution modulo~\(23\cdot67\). Assume therefore that~\(p=73\) and write \(k=3k_0+1\). Then \((X,Y)=(3^{k_0+1} \cdot 73,3 \cdot 73^2 z)\) is an integral point on the elliptic curve
\[
Y^2=X^3-2^33^{2} \cdot 73^3.
\]
This gives rise to the unique solution~\(3^4-2^3=73\),  which completes the proof of the proposition.

\end{proof}

Utilizing Proposition \ref{prop:local_obstructions_2-3-p-n=2}, we can again appeal to Cremona's SAGE package to compute primitive solutions to equation~(\ref{two}). For each set~\(S=\{2,3,p\}\), we have tabulated these solutions in  the file
\verb"http://www.math.ubc.ca/~bennett/2-3-p.pdf".
By quick examination of this table for~\(p=23\) and~\(p=47\), we immediately deduce the following result about equations that remained unsolved in Proposition 3.3 of Terai \cite{Ter14} (but have been recently solved via rather different methods by Deng \cite{Den15}).

\begin{proposition}\label{prop:Terai}
The only solutions to \(x^2+23^m=12^n\) and \(x^2+47^m=24^n\) are \((x,m,n)=(\pm11,1,2)\) and \((x,m,n)=(\pm23,1,2)\) respectively.
\end{proposition}

\section{The case $n=3$}\label{s:n=3}

We now deal with equation~(\ref{enn}) when \(n=3\) using a \((p,p,3)\) Frey-Hellegouarch curve approach. The corresponding equations to treat are of the shape
\begin{equation}\label{eq:n=3}
x+y = w z^3
\end{equation}
where $x, y$ and $w$ are pairwise coprime $S$-units. We may assume, without loss of generality, that $w$ is cubefree and positive and that we have \(x\not\equiv0\mod{3}\) and \(y\not\equiv2\mod{3}\). With the notation of \S\ref{ss:nn3}, we attach to such a solution the elliptic curve~\(E=E_{n,n,3}^{x,y,w}(1,1,z)\) :
\[
E \; \colon \;  Y^2+3wzXY+w^2yY=X^3.
\]
As in Section~\ref{s:n=2}, all that remains to do to solve equation~(\ref{eq:n=3}) is to check for an equality
\[
j(E)=j(F),\quad\text{where \(j(E)=3^3\frac{(x+y)(9x+y)^3}{xy^3}\)}
\]
and \(j(F)\) denote the \(j\)-invariants of~\(E\) and~\(F\) respectively, with~\(F\) ranging over all representatives of the isomorphism classes of elliptic having good reduction outside~\(S\cup\{3\}\) (and a nontrivial \(3\)-torsion subgroup). To circumvent the difficulty of computing representatives and make this approach work for a broader list of sets~\(S\) (including some for which we do not know a complete list of such representatives), we also make use of the precise formula for the conductor of~\(E\) given by Lemma~\ref{lem:nn3_arith}, namely~\(N(E)=3^{\alpha} \rad{3}(xy) (\rad{3}(w))^2\) where
\[
\alpha = 
\begin{cases}
2 & \text{if $w^2y-3wz\equiv-2\mod{9}$} \\
3 & \text{if $w^2y-3wz\equiv1$ or $4\mod{9}$} \\
4 & \text{if $\ord{3}(y)=1$} \\
3 & \text{if $\ord{3}(y)=2$} \\
0 & \text{if $\ord{3}(y)=3$} \\
1 & \text{if $\ord{3}(y)>3$} \\
5 & \text{if $3\mid w$}.
\end{cases}
\]
As in the previous section, we now apply this approach to the sets~\(S=\{2,3,5,7\}\) and~\(S=\{2,3,p\}\) with~\(p\) prime, \(11\le p\le 100\).

\subsection{The case $S=\{2,3,5,7\}$}
Specializing to~\(S=\{2,3,5,7\}\), we \emph{a priori} need to consider all conductors of the shape
\[
N(E)=2^{\delta_2}\cdot3^{\alpha}\cdot5^{\delta_5}\cdot7^{\delta_7}
\]
with \(\delta_i(N)=\ord{i}\in\{0,1,2\}\), \(i=2,5,7\) and \(\alpha=\ord{3}(N)\in\{0,1,2,3,4,5\}\). However, we have, as before, the following result.
\begin{lemma}
In each case \(N(E)<350000\).
\end{lemma}
\begin{proof}
It is easy to check that we have the desired inequality for $N(E)$, unless
\[
(\alpha,\delta_2,\delta_5,\delta_7)\in\{(5,1,2,2),(5,2,2,2),(4,2,2,2)\}.
\]
Among these possibilities, only the case \((\alpha,\delta_2,\delta_5,\delta_7)=(5,1,2,2)\) may correspond to solutions. Since each of the equations 
\[
2^k\pm1=3^{\delta_3}5^{\delta_5}7^{\delta_7}z^3,\quad\text{with each } \delta_i \in \{1,2 \}
\]
is insoluble modulo~\(840\), we obtain the stated result.
\end{proof}

As explained previously, we can therefore appeal to Cremona's table of elliptic curves to solve equation~(\ref{eq:n=3}) and thus~(\ref{enn}) with~\(n=3\) and~\(S=\{2,3,5,7\}\). It turns out that there are exactly~\(207\) triples~\((x,y,z)\) such that~\(x+y=z^3\) with $x$ and $y$ \(\{2,3,5,7\}\)-units, with \(\gcd(x,y)\) cubefree and, say, \(x\ge\abs{y}>0\),  and $z$ a positive integer. They are listed in the file
\verb"http://www.math.ubc.ca/~bennett/2-3-5-7.pdf".
From there we easily extract the latter solutions to this equation for~\(S=\{2,3\}\) and~\(S=\{3,5,7\}\). We list them here for later use.
\begin{proposition}
The only primitive solutions to equation~(\ref{enn}) with~\(n=3\), \(S=\{2,3\}\) and, say, $x \ge |y| > 0$ are given by
\[
\begin{array}{c}
(x,y) = (2,-1),\ (3,-2),\ (4,-3),\ (4,4),\ (6,2),\ (9,-8),\ (9,-1),\\
(12,-4),\ (18,9),\ (24,3),\ (36,-9)\quad\text{and}\quad (128,-3). \\
\end{array}
\]
\end{proposition}

\begin{proposition}
The only primitive solutions to equation~(\ref{enn}) with~\(n=3\), \(S=\{3,5,7\}\) and, say, $x \ge |y| > 0$ are given by
\[
\begin{array}{c}
(x,y) = (5,3),\ (7,1),\ (9,-1),\ (15,-7),\ (35,-27),\ (49,15),\ (63,1),\ (189,-125),\\
(225,-9),\ (441,-225),\ (1225,-225),\ (1875,-147)\quad\text{and}\quad (3969,-1225). \\
\end{array}
\]
\end{proposition}

\subsection{The case $S=\{2,3,p\}$} We now turn our attention to the case~\(S=\{2,3,p\}\) where $11 \leq p < 100$  is  prime. Once again, we \emph{a priori} need to consider all conductors of the shape
\[
N(E)=2^{\delta_2}\cdot3^{\alpha}\cdot p^{\delta_p}
\]
with \(\delta_i=\ord{i}(N)\in\{0,1,2\}\), \(i=2,p\) and \(\alpha=\ord{3}(N)\in\{0,1,2,3,4,5\}\). However, we have the following result, the proof of which, it being quite similar to that of Proposition~\ref{prop:local_obstructions_2-3-p-n=2}, we  omit for the sake of concision.
\begin{proposition}
We have~\(N(E)\le 350000\) unless~\((x,y,w,z)\) corresponds to a one of the following equations
\[
2^7+1=3\cdot43,\quad 3\cdot2^4-1=47,\quad 3\cdot2^5+1=97,
\]
\[
2^6+3=67\quad\text{and}\quad  2^6-3=61.
\]
\end{proposition}
Combining this proposition with Cremona's tables of elliptic curves then allows us to compute all the solutions to~(\ref{enn}) with~\(n=3\), \(S=\{2,3,p\}\) and~\(p\) as above. We list them in the file
\verb"http://www.math.ubc.ca/~bennett/2-3-p.pdf".

\section{The general equation}\label{s:all_solutions}

In this last section, we completely solve equation~(\ref{enn}) for two specific sets of primes, namely \(S=\{2,3\}\) and~\(S=\{3,5,7\}\), using a variety of Frey-Hellegouarch curves, level-lowering and heavy computations involving modular forms and Thue-Mahler equations. It is only through careful combination of a variety of Frey-Hellegouarch curves, in conjunction with local arguments, that we are able to reduce these problems to a feasible collection of Thue-Mahler equations (in our case, all of degree $5$). This is, in essence, the main feature of our approach that distinguishes it from one purely reliant upon lower bounds for linear forms in logarithms. This latter method is, in our opinion, at least with current technology, impractical for explicitly solving equation (\ref{enn}) for any set $S$ with at least two elements.

\subsection{The case $S=\{2,3\}$}
We prove the following result.
\begin{theorem} \label{oof}
The only primitive solutions to equation (\ref{enn}) with \(S=\{2,3\}\) and, say, $x \ge |y| > 0$ and~\(z>0\) are given by the following infinite families
\[
\begin{array}{l}
(x,y,z,n)=(2,-1,1,n), (3,-2,1,n), (4,-3,1,n), (9,-8,1,n), (2^{n-1},2^{n-1},2,n),\\ 
 (3\cdot2^{n-2},2^{n-2},2,n), (3\cdot2^{n-1},-2^{n-1},2,n), (2\cdot3^{n-1},3^{n-1},3,n),\\ 
 (2^2\cdot3^{n-1},-3^{n-1},3,n), (2^3\cdot3^{n-2},3^{n-2},3,n),\quad\text{ all with }n\ge2, \\
(x,y,z,n)=(3^2\cdot2^{n-3},-2^{n-3},2,n)\quad\text{for }n\ge3 \\
\end{array}
\]
and by
$$
\begin{array}{l}
(x,y,z,n) = (16,9,5,2),\ (18,-2,4,2),\ (24,1,5,2),\ (27,-2,5,2),\ (81,-32,7,2),\\ 
(48,1,7,2),\ (128,-3,5,3),\ (288,1,17,2) \mbox{ and }  (486,-2,22,2).
\end{array}
$$
\end{theorem}

The cases where \(n\le 4\) were covered in the previous two sections. We may therefore assume, without loss of generality, that  $n \geq 5$ is prime. The corresponding equation to treat is
\begin{equation} \label{fridge}
x+y=wz^n
\end{equation}
with $x, y$ and $w$ coprime $\{2, 3 \}$-units and \(w\)  \(n\)th-power free.

Assume first that \(\ord{2}(xyw)\ge4\). Since precisely one of $x, y$ and $wz$ is even, we may suppose, without loss of generality,  that \(x\equiv-1\mod{4}\) and that 
$$
\max \left\{ \mbox{ord}_2 (y) , \mbox{ord}_2 (w) \right\} \geq 4.
$$
We write $(Aa^n,Bb^n)=(x,y)$ or $(x,-wz^n)$ if $y$ or $w$ is even, respectively, and set 
 \(Cc^n=Aa^n+Bb^n\), for  $A, B$ and $C$  \(n\)th-power free integers. With the notation of~\S\ref{ss:nnn}, considering the elliptic curve~\(E=E_{n,n,n}^{A,B,C}(a,b,c)\), we thus have that $n$ fails to divide $ABC$ and~\(\ord{2}(ABC)\ge4\), and hence, in the notation of Proposition~\ref{nnn}, \(E\sim_nf\) where \(f\) is a weight \(2\) newform of level~\(N_0\mid 6\). The resulting contradiction implies that necessarily~\(\ord{2}(xyw)\le3\).

Assume next that \(\ord{3}(xyw)\ge4\). Without loss of generality, we may suppose \(x\not\equiv0\mod{3}\), \(y\not\equiv2\mod{3}\) and that
$$
\max \left\{ \mbox{ord}_3 (y) , \mbox{ord}_3 (w) \right\} \geq 4.
$$
We define
$$
(Aa^n,Bb^n,Cc^3)= (-wz^n,y,-x) \; \mbox{ or } \; (-x,wz^n,y),
$$
depending on whether $3 \mid y$ or $w$, respectively, where, in either case, $A$ and $B$ are chosen to be $n$th-power free and $C$ is cubefree. In both cases, we have \(n\nmid ABC\) and \(\left|ab\right|\ge\left|z\right|\). 

If \(\left|z\right|\ge2\), then with the notation of Proposition~\ref{lem:nn3_LL-1}, we have \(E_{n,n,3}^{A,B,C}(a,b,c)\sim_nf\) where \(f\) is a weight \(2\) newform of level~\(N_0\mid 12\) and hence a contradiction. We therefore have~\(z=\pm1\) and we are led to solve \(3^k-2^l=\pm1\) with \(k\ge4\). Both equations have no solution for~\(l\le2\) and, by reducing modulo~\(8\), we see that the equation~\(3^k-2^l=-1\) has no solution for~\(l\ge3\) as well. If~\(3^k-2^l=+1\) and~\(l\ge3\), then \(k\) is necessarily even, whereby both \(3^{k/2}-1\) and~\(3^{k/2}+1\) are powers of~\(2\). Hence~\(k<4\) and this is again a contradiction. 

We are left, then, to consider  equation (\ref{fridge}) with
$$
\max \left\{ \mbox{ord}_2 (xyw) , \mbox{ord}_3 (xyw) \right\} \leq 3.
$$
A short calculation leads to the families indicated in Theorem \ref{oof}.

\subsection{The case $S=\{3,5,7\}$}
The following result is the main theorem of the paper. As we shall observe, despite the apparently small size of $S$ and its elements, the computations involved here are really approaching the limits of current ``off the shelf'' technology (though coming refinements in computational tools for modular forms will alleviate this somewhat).
\begin{theorem}\label{thm:main}
The only primitive solutions to equation (\ref{enn}) with \(S=\{ 3, 5, 7\}\) and $x > |y| > 0$  are given by
$$
\begin{array}{l}
(x,y) = (3,1),  (5,-1), (5,3), (7,-3), (7,1), (9,-5), (9, -1), (9,7), (15,-7), (15,1), \\
(21,-5), (21,15), (25,-21), (25,-9), (25,7), (27,5), (35, -27),  (35, -3), (35, 1),  \\ 
(49,-45), (49,15), (63,1), (81,-49), (105,-5), (125,3),  (135,-35), (135,-7),\\ 
(147,-3), (175,21), (175,81), (189,-125), (189,7), (225,-9), (343,-243), \\ 
(375,-343), (405,-5), (441,-225), (625,-49), (675,1), (729,-245), (1029,-5), \\ 
(1225,-225), (1323,-27), (1875,-147), (3375,2401),  (3969,-1225), (3969,-125), \\ 
(9375,1029), (10125,-125), (15625,-1701), (50625,-3969),  (59535,1),\\  
(540225,-2401), (688905,-5), (4782969,4375)  \mbox{ and } (24310125,-10125).
\end{array}
$$
\end{theorem}
Before proving our main theorem, we first state some useful preliminary results starting with a standard factorization lemma.
\begin{lemma}\label{lem:factorization}
Let  $x$ and $y$ be coprime integers and~\(n\) an odd prime number. If we write 
$$
\phi_n(x,y)=\frac{x^n+y^n}{x+y},
$$
then we have that
$$
\gcd\left(x+y,\phi_n(x,y)\right) \in \{ 1, n \}
$$
and, moreover,
\[
\gcd\left(x+y,\phi_n(x,y)\right)=n\Leftrightarrow n\mid x^n+y^n \Leftrightarrow n\mid\phi_n(x,y)\Leftrightarrow n\mid x+y.
\]
Further, \(n^2\nmid \phi_n(x,y)\) and if \(\ell\) is a prime dividing~\(x^n+y^n\), then either \(\ell\mid x+y\) or \(\ell\equiv 1\mod{n}\).
\end{lemma}
\begin{proof} Apart from its very last assertion, this lemma is proved in~\cite[Lemma~2.1]{DaSi14}. Suppose that~\(\ell\) is a prime dividing~\(x^n+y^n\). Then, by coprimality, \(\ell\) does not divide~\(y\). We may thus find an integer~\(y'\)  such that~\(yy'\equiv-1\mod{\ell}\), whence  \((xy')^n\equiv1\mod{\ell}\). If \(xy'\equiv1\mod{\ell}\), it follows that
$$
xy'+yy'= y' (x+y) \equiv 0 \mod{l},
$$
so that $l \mid x+y$. Otherwise, since $n$ is prime, $xy'$ has order $n$ modulo $l$, so that $n \mid l-1$, i.e. $l \equiv 1 \mod{n}$.
\end{proof}

This lemma is a key tool in the proof of the following result.
\begin{lemma}\label{lem:factorization2}
Let~\(x,y\) be coprime nonzero integers and~\(n\ge5\) be a prime number. Then
\[
P(x^n+y^n) \geq 11,
\]
unless~\(\abs{x}=\abs{y}=1\). 
\end{lemma}
\begin{proof}
Assume~\(P(x^n+y^n)\le7\) and \(\abs{x}\not=\abs{y}\). Then, with the above notation, we have that~\(\abs{\phi_n(x,y)}>1\), since, if  \(x^n+y^n=\pm(x+y)\), then necessarily ~\(\abs{x}=\abs{y}\). It follows that there exists a prime $l \in \{ 2, 3, 5, 7 \}$ such that $l \mid \phi_n (x,y)$ and hence, by Lemma \ref{lem:factorization}, $l \mid x+y$, whereby $l=n \in \{ 5, 7 \}$. We thus have
\[
x^n+y^n=\pm n(x+y),\quad\text{for \(n=5\) or~\(7\)},
\]
which again has no solutions with~\(\abs{x}\not=\abs{y}\). This contradiction completes the proof of  the lemma.
\end{proof}

The next result will be of use in proving a special case of Theorem~\ref{thm:main} (see Proposition~\ref{prop:y_pm_1}).
\begin{proposition}\label{prop:Stormer}
The only solutions to \(C^4-1=p^{\alpha}q^{\beta}z^n\) and \(D^2-1=p^{\alpha}z^n\) for \(C,D\) and \(z\) positive integers with \(z\) even, \(n\geq5\) prime, and \(\{p,q\}\subset\{3,5,7\}\) are with
\[
n=5,\quad C=7\quad\text{and}\quad D\in\{15,17\}.
\]
\end{proposition}
\begin{proof}
We first deal with \(C^4-1=p^{\alpha}q^{\beta}z^n\). Since \(z\) is even, we have that \(C\) is odd and by factorization (up to permutation of \(p\) and \(q\)) that either
$$
C^2+1=2z_1^n, \; C^2+1=2p^{\alpha}z_1^n \; \mbox{ or } \; 
C^2+1=2p^{\alpha}q^{\beta}z_1^n,
$$
for some odd integer \(z_1\).
By a classical result of St\"ormer (\cite[p.~168]{Sto99}), the first case holds only for $z_1=C=1$. Similarly in the third case, we deduce that \(C^2-1=2^{n-1}z_2^n\) for some nonzero integer $z_2$ which in turn contradicts Corollary~1.4 of~\cite{Ben04}. We therefore end up in the situation where
\[
C^2+1=2p^{\alpha}z_1^n\quad\text{and}\quad C^2-1=2^{n-1}q^{\beta}z_2^n,
\]
for some positive integer~\(z_2\). By factorization, ~\(C\pm1=2q^{\beta}w_1^n\) and~\(C\mp1=2^{n-2}w_2^n\) for some positive integers $w_1$ and $w_2$, whereby we have
\[
q^{\beta}w_1^n-2^{n-3}w_2^n=\pm1.
\]
We may thus appeal to Theorem 1.1 of \cite{BGMP06} with (in the notation of that paper) \(S=\{2,q\}\) to conclude that \(n=5\) and~\(w_1=w_2=1\), whence  \(C=7\) as claimed.

We now turn our attention to the equation \(D^2-1=p^{\alpha}z^n\). By factorization, for some positive integers $w_1$ and $w_2$, we have either
\[
w_1^n-2^{n-2}p^{\alpha}w_2^{n}=\pm1\quad\text{or}\quad p^{\alpha}w_1^n-2^{n-2}w_2^{n}=\pm1.
\]
According to \emph{loc. cit.} applied to \(S=\{2,p\}\),  the former equation has no such solutions, whereas the only solutions to the latter are with~\(n=5\), \(p^{\alpha}\in\{7,3^2\}\) and \(w_1=w_2=1\). This gives rise to~\(D\in\{15,17\}\), as claimed. 
\end{proof}

With those preliminary results in hands, we now turn to the actual proof of Theorem~\ref{thm:main}. Once again, recall that we have to solve
\begin{equation}\label{eq:357mod}
x+y=wz^n,
\end{equation}
where \(x,y,w\) are coprime \(\{3,5,7\}\)-units with~\(x,w\) positive, \(n\geq5\) prime and~\(z\)  nonzero.

Our first result is concerned with the case \(y=\pm1\) in equation~(\ref{eq:357mod}) above.

\begin{proposition}\label{prop:y_pm_1}
Let \(x\) and \(w\) be coprime positive \(\{3,5,7\}\)-units, \(y=\pm1\) and let \(n\geq5\) be a prime number such that \(x+y=wz^n\) for some positive integer~\(z\). Then, \(y=-1\), \(n=5\), \(z=2\) and 
\[
(x,w)=(7^4,3\cdot5^2)\quad\text{or}\quad (3^2\cdot5^2,7).
\]
\end{proposition}
\begin{proof}
According to Lemma~\ref{lem:factorization2}, we may assume, without loss of generality, that \(w\not=1\). Similarly by applying~\cite[Thm.~1.1]{BGMP06} to each subset of~\(\{3,5,7\}\) of cardinality~\(2\), we may also assume that \(\rad{}(xw)=3\cdot5\cdot7\). Besides, sieving modulo~\(2^4\rad{}(w)\) shows that we necessarily have~\(y=-1\) and that either~\(x\) is a \(4\)th power and \(w\) has two distinct prime factors (in \(\{3,5,7\}\)), or that \(x\) is a square and \(w\) has only one prime factor. We finally conclude using Proposition~\ref{prop:Stormer}.
\end{proof}

According to Lemma~\ref{lem:factorization2} and Proposition~\ref{prop:y_pm_1} above, in order to prove Theorem~\ref{thm:main}, it remains to solve each equation of the shape
\begin{align} \label{eq6}
3^{\alpha}5^{\beta} + (-1)^\delta 7^{\gamma} & =  z^n,\quad\text{with }(\alpha,\beta)\not\equiv(0,0)\mod{n}\text{ and }\gamma\not\equiv0\mod{n} \\ \label{eq14}
3^{\alpha}7^{\gamma} + (-1)^\delta 5^{\beta} & =  z^n,\quad\text{with }(\alpha,\gamma)\not\equiv(0,0)\mod{n}\text{ and }\beta\not\equiv0\mod{n} \\ \label{eq15}
5^{\beta}7^{\gamma} + (-1)^\delta 3^{\alpha} & =  z^n,\quad\text{with }(\beta,\gamma)\not\equiv(0,0)\mod{n}\text{ and }\alpha\not\equiv0\mod{n} \\ \label{eq16}
3^{\alpha} + (-1)^\delta 5^{\beta} & =  7^{\gamma}z^n,\quad\text{with }\alpha,\beta>0\text{ and }0<\gamma\leq n-1 \\ \label{eq17}
3^{\alpha} + (-1)^\delta 7^{\gamma} & =  5^{\beta}z^n,\quad\text{with }\alpha,\gamma>0\text{ and }0<\beta\leq n-1 \\ \label{eq18}
5^{\beta} + (-1)^\delta 7^{\gamma} & =  3^{\alpha}z^n,\quad\text{with }\beta,\gamma>0\text{ and }0<\alpha\leq n-1
\end{align} 
where $\alpha, \beta$ and $\gamma$ are nonnegative integers, \(n\geq5\) is prime and $\delta \in \{ 0, 1 \}$. 

In the remainder of this section, we prove the following precise result on these equations. Combining it with Proposition~\ref{prop:y_pm_1} and results from Sections~\ref{s:n=2} and~\ref{s:n=3}  completes the proof of Theorem~\ref{thm:main}.
\begin{theorem} \label{boss}
The solutions to equations (\ref{eq6}) -- (\ref{eq18}) in nonnegative integers $\alpha, \beta$ and $\gamma$, prime $n \geq 5$ and $\delta \in \{ 0, 1 \}$ correspond to the identities
$$
\begin{array}{c}
 2^5 = 3\cdot5^3-7^3 = 3^4-7^2 = 5\cdot7-3 = 3^3+5 = 5^2+7,  \\ 
2^{10} = 3 \cdot 7^3 - 5\quad\text{and}\quad  2^7  = 5^3+3 = 3^3\cdot5-7. \\
\end{array}
$$
\end{theorem}

Suppose that we have a solution to one of equations  (\ref{eq6}) -- (\ref{eq18}) in nonnegative integers $\alpha, \beta$ and $\gamma$, prime $n \geq 5$ and $\delta \in \{ 0, 1 \}$, and rewrite this in the shape (\ref{eq-nn3}) for suitable choices of $A, B, C, a, b$ and $c$. Applying Proposition \ref{lem:nn3_LL-1}, we thus have that $E \sim_n f$ for some weight $2$ cuspidal newform of level $N_0$ where, crudely, we have that $N_0 \mid 3^5 \cdot 5^2 \cdot 7^2$. Since $2 \mid z$ and hence necessarily $2 \mid ab$, we may apply congruence (\ref{tech2}) with $l=2$ to conclude that
$$
n \leq \left(3 + 2 \sqrt{2} \right)^{g_0(N_0)} \leq \left(3 + 2 \sqrt{2} \right)^{(N_0+1)/12} < 10^{18991},
$$
where \(g_0(N_0)\) denotes the number of cuspidal weight-\(2\) newforms of level~\(N_0\).

In what follows, we will in fact show that this upper bound can, through somewhat refined arguments using various Frey-Hellegouarch curves, local computations and Thue(-Mahler) solvers, be replaced by the assertion that $n \in \{ 5, 7 \}$, corresponding to the solutions noted in Theorem~\ref{boss}.
In the case $n=5$, an approach via Frey-Hellegouarch curves, while theoretically of value, in practice appears to work poorly. Instead, we will use MAGMA code for solving Thue-Mahler equations due to K. Hambrook \cite{Ham}; documentation for this may be found at 

\vskip0.6ex
\hskip8ex \verb"http://www.math.ubc.ca/~bennett/hambrook-thesis-2011.pdf".

\vskip0.5ex \noindent The result we deduce by appealing to this code is the following :
\begin{proposition}
The solutions to equations (\ref{eq6}) -- (\ref{eq18}) in nonnegative integers $\alpha, \beta, \gamma$ and $\delta \in \{ 0, 1 \}$, with \(n=5\) correspond to the identities
$$
2^5 = 3\cdot5^3-7^3 = 3^4-7^2 = 5\cdot7-3 = 3^3+5 = 5^2+7\quad\text{and}\quad 2^{10} = 3 \cdot 7^3 - 5. 
$$
\end{proposition}
\begin{proof}
According to the shapes of the equations and the noted restrictions on~\(\alpha,\beta\) and \(\gamma\), it suffices to solve the following Thue-Mahler equations (some of which may be treated locally)
\begin{align*} 
z^5-3^a5^by^5&=7^c,\quad\text{with }0\leq a,b\leq4\text{ and } (a,b)\not=(0,0) \\
z^5-3^a7^cy^5&=5^b,\quad\text{with }0\leq a,c\leq4\text{ and } (a,c)\not=(0,0) \\
z^5-5^b7^cy^5&=3^a,\quad\text{with }0\leq b,c\leq4\text{ and } (b,c)\not=(0,0) \\
7^cz^5-5^by^5&=3^a,\quad\text{with }0< b,c\leq4 \\
3^az^5-5^by^5&=7^c,\quad\text{with }0< a,b\leq4. \\
\end{align*}
This is easily achieved using Hambrook's code and leads to the solutions mentioned.
\end{proof}

It is worth noting, at this point in the proceedings, that, at least as currently implemented, the Thue-Mahler solver we are using  has severe difficulties with equations of degree $7$ and higher. In particular, we will need to work very carefully in order to avoid its use for  the remaining cases under consideration.

We will now deal with each of equations~(\ref{eq6})--(\ref{eq18}) in turn, assuming~\(n\geq 7\) and that we have a solution satisfying the conditions mentioned. We will make repeated use of Proposition~\ref{lem:nn3_LL-1}, explicitly using MAGMA to calculate newforms of all relevant levels $N_0$.
 
The results are as follows. We list in Table~\ref{table:levels} the triples $(\delta_3,\delta_5,\delta_7)$ of interest to us for which the space of weight $2$ cuspidal newforms  of level
\begin{equation} \label{level}
N_0 = 3^{\delta_3} \cdot 5^{\delta_5} \cdot 7^{\delta_7}
\end{equation}
	is nontrivial, together with the list of all primes $n\geq7$ for which there exists at least one form $f$ of level $N_0$ satisfying both  (\ref{tech2}) with $l=2$, and the congruences of Proposition~\ref{lem:nn3_LL-1} for all primes $3 \leq l < 50$. The MAGMA code used for this computation (when \(N_0<10000\)) is available at
\verb"http://www.math.ubc.ca/~bennett/nn3.m".

For the larger levels 
$N_0=3^4\cdot5^2\cdot7=14175$ and $N_0=3^4\cdot5\cdot7^2=19845$,
we have first used Wiese's MAGMA function \verb"Decomposition" (from his package  \verb"ArtinAlgebras" available on his homepage) to compute the characteristic polynomials of the first few Fourier coefficients. The output can be found in 

\vskip0.6ex
\hskip14ex \verb"http://www.math.ubc.ca/~bennett/14175.m" 

\noindent and in  

\hskip14ex  \verb"http://www.math.ubc.ca/~bennett/19845.m",

\noindent respectively.

\begin{table}[h!]
\begin{tabular}{|c|c|}\hline
	$(\delta_3,\delta_5,\delta_7)$ & $n$ \\ \hline
	$(0,0,2),\ (0,2,1),\ (1,0,1),\ (1,1,0),\ (1,2,0),\ (1,2,1),$ & \mbox{ none } \\ 
	$(3,0,0),\ (3,0,1),\ (3,0,2),\ (4,0,0),\ (4,1,0),\ (5,0,1)$ & \\ \hline
	$(0,1,2),\ (0,2,2),\ (1,0,2),\ (1,1,2),$  & 7 \\
	$(4,0,2),\ (4,2,1),\ (5,1,0)$ & \\ \hline
	$(1,2,2),\ (3,2,1),\ (4,1,2)$ & 7, 11 \\ \hline
	$(5,0,0)$ & 17 \\ \hline
	$(5,1,1)$ & 7, 17 \\ \hline
\end{tabular}
\vskip1ex
\caption{Triples $(\delta_3,\delta_5,\delta_7)$ and corresponding values of $n$\label{table:levels}}
\end{table}
In the next six subsections we deal with each equation in turn. Full details on the computations can be found in the file

\vskip0.6ex
\hskip8ex  \verb"http://www.math.ubc.ca/~bennett/Last_Equations.pdf"

\subsubsection{The equation $3^{\alpha}5^{\beta} + (-1)^\delta 7^{\gamma} =  z^n$}
By reducing the equation modulo \(8\), we see that \(\alpha\) and \(\beta\) have the same parity, and that \(\alpha\) is odd if and only if we have \(\delta\equiv\gamma\mod{2}\). We begin by assuming that $\alpha$ and $\beta$ are even, and $n\geq 7$. If, further, $\beta > 0$, we 
consider the curve
\[
E=E_{(3),n,n,2}^{(-1)^{\delta+1}7^{\gamma_0},1,1}(7^{\gamma_1},z,(-1)^{\alpha/2}3^{\alpha/2}\cdot5^{\beta/2}),
\]
where \(\gamma=n\gamma_1+\gamma_0\), with \(0\leq\gamma_0\leq n-1\). It has good reduction at \(5\) and is of the shape
\[
E \; \; : \; \;  Y^2+XY=X^3+AX^2+BX
\]
where \(A,B\in\mathbb{Z}\) with \(A\equiv1\mod{5}\) and \(B\not\equiv0\mod{5}\), whereby  \(a_5(E)\in\{\pm2,\pm4\}\). On the other hand, Proposition~\ref{nn2}, \(E\sim_n f\) where~\(f\) is a weight-\(2\) newform of level~\(14\). Since there is a unique such newform and it satisfies \(c_5=0\), we obtain  a contradiction from~(\ref{tech2}) for \(l=5\). Assume next that $\alpha$ is even, $\beta=0$ and $n = 7$. If \(\gamma\geq 2\), then we have \(3^{\alpha}\equiv z^7\mod{49}\) and, since \(z\not\equiv0\mod{7}\),  \(3^{6\alpha}\equiv1\mod{49}\). This leads to the conclusion that \(\alpha\equiv0\mod{7}\) and hence a contradiction. If \(\gamma=1\), then we have \(\delta=0\) and the equation to solve is \(3^{\alpha}+7=z^7\). However we check that there is no value of~\(\alpha\mod{42}\) such that
\[
(3^{\alpha}+7)^6\equiv1\mod{49}\quad\text{and}\quad (3^{\alpha}+7)^6\equiv0\text{ or }1\mod{43},
\]
and hence obtain a contradiction in this case as well. We therefore may assume that $\alpha$ is even, $\beta=0$ and  \(n\geq 11\), and consider the curve
\[
E=E_{n,n,3}^{1,-3^{\alpha_0},(-1)^{\delta}7^{\gamma_0}}(z,3^{\alpha_1},7^{\gamma_1}),
\]
where
$\alpha=n\alpha_1+\alpha_0$,   with   $0<\alpha_0\leq n-1$  and 
$\gamma=3\gamma_1+\gamma_0$,  with $0\leq\gamma_0\leq 2$.
Then, by Proposition~\ref{lem:nn3_LL-1}, \(E\sim_n f\) where \(f\) is a weight \(2\) newform of level \(N_0\) with \(N_0\mid 3\cdot7^2\) or \(3^3\mid N_0\mid 3^3\cdot7^2\) if \(\alpha\geq3\) or \(\alpha=2\) respectively. Since \(n\geq11\), we therefore reach a contradiction upon appealing to Table~\ref{table:levels}.

Next, suppose that $\alpha$ and $\beta$ are odd. If $n=7$ and \(\gamma\geq2\), we simply note that there is no solution modulo~\(2^3\cdot7^2\cdot29\cdot43\cdot71\). If $n=7$ and \(\gamma=1\), then \(\delta=1\) and modulo~\(2^3\cdot7^2\cdot29\cdot43\cdot113\), we have that \(\alpha\equiv3\mod{7}\) and \(\beta\equiv1\mod{7}\). We  are therefore led  to solve the Thue equation
\[
z^7-3^35y^7=7
\]
using PARI/GP. This gives rise to a unique solution to our equation, namely \(3^3\cdot5-7=2^7\). If, however, we assume that $n \geq 11$, we begin by considering the following \((n,n,3)\) Frey-Hellegouarch curve
\[
E=E_{n,n,3}^{1,-3^{\alpha_0}5^{\beta_0},(-1)^{\delta}7^{\gamma_0}}(z,3^{\alpha_1}5^{\beta_1},7^{\gamma_1})
\] 
where
\[
\left\{\begin{array}{ll}
\alpha=n\alpha_1+\alpha_0, & 0\leq\alpha_0\leq n-1 \\
\beta=n\beta_1+\beta_0, & 0\leq\beta_0\leq n-1 \\
\gamma=3\gamma_1+\gamma_0, & 0\leq\gamma_0\leq 2. \\
\end{array}
\right.
\]
Then, by Proposition~\ref{lem:nn3_LL-1}, \(E\sim_n f\) where $f$ has level \(N_0\) with \(N_0\mid 3\cdot5\cdot7^2\) or $N_0 = 3^4 \cdot N_1$, where $N_1\mid 5\cdot7^2$, if \(\alpha\geq3\) or \(\alpha=1\) respectively. According to Table~\ref{table:levels}, since \(n\geq11\), we necessarily have level \(N_0=3^4\cdot5\cdot7^2\), whence
\[
\alpha=1,\quad \beta\equiv1\mod{2},\quad n=11\quad\text{and}\quad \beta\not\equiv0\mod{11}.
\] 
To treat this remaining case, we now consider the following \((n,n,n)\) Frey-Hellegouarch curve 
\[
E=E_{11,11,11}^{3\cdot5^{\beta_0},-1,(-1)^{\delta+1}7^{\gamma_0}}(5^{\beta_1},z,7^{\gamma_1}),
\] 
where  \(\beta=11\beta_1+\beta_0\) and \(\gamma=11\gamma_1+\gamma_0\) with \(0<\beta_0,\gamma_0\leq 10\) (recall that \(\beta,\gamma\not\equiv0\mod{11}\)). According to Proposition~\ref{nnn}, \(E\) arises from a weight-\(2\) newform \(f\) of level \(210\) and \(f\) corresponds to an elliptic curve \(F/\mathbb{Q}\) such that
\[
a_q(F)\not\equiv\pm2\mod{11},\quad\text{for $q=23$ and $67$}.
\]
This implies that \(F\) has good reduction at \(23\) and \(67\), or, in other words, that neither \(23\) nor \(67\) divides \(z\). We may thus sieve locally at these primes to show that there is no triple \((\beta,\delta,\gamma)\) (with \(\beta\) odd and \(\delta\equiv\gamma\mod{2}\)) such that
\[
3\cdot5^{\beta}+(-1)^{\delta}7^{\gamma}\in \big(\mathbb{F}_q^{\times}\big)^{11}
\]
simultaneously for \(q=23\) and \(67\). This leads to the desired contradiction, whereby we may conclude that the only solution to~(\ref{eq6}) for \(n\geq 7\) prime corresponds to $3^3\cdot5-7=2^7$.


\subsubsection{The equation $3^{\alpha}7^{\gamma}+(-1)^{\delta}5^{\beta}=z^n$}
By reducing modulo \(8\), we see that \(\alpha\) and \(\beta\) necessarily have the same parity, and that \(\alpha\) is odd precisely when we have \(\delta\equiv\gamma\mod{2}\). Assume first that $\alpha$ and $\beta$ are even, and $n\geq 7$. Let us consider 
\[
E=E_{(3),n,n,2}^{-3^{\alpha_0}7^{\gamma_0},1,(-1)^{\delta}}(3^{\alpha_1}7^{\gamma_1},z,(-1)^{\delta}5^{\beta/2}),
\]
where \(\alpha=n\alpha_1+\alpha_0\), \(\gamma=n\gamma_1+\gamma_0\), with \(0\leq \alpha_0,\gamma_0\leq n-1\).
Since \(\beta>0\), $E$ has good reduction at \(5\). Moreover, if~\(\delta=0\), then it satisfies \(a_5(E)=\pm4\). From Proposition~\ref{nn2}, \(E\) arises modulo \(n\) from a weight \(2\) newform of level \(42\) (recall that \(\alpha>0\)). But there is a unique newform at level~\(42\)  corresponding to an elliptic curve \(F/\mathbb{Q}\) with \(a_5(F)=-2\), a contradiction. It  follows that we have \(\delta=1\), which in turn implies \(\gamma\) even. We may thus consider instead
\[
E=E_{(3),n,n,2}^{5^{\beta_0},1,1}(5^{\beta_1},z,(-1)^{(\alpha+\gamma)/2}3^{\alpha/2}7^{\gamma/2}),
\]
where \(\beta=n\beta_1+\beta_0\) with \(0<\beta_0\leq n-1\). Then, by Proposition~\ref{nn2}, \(E\) arises from a weight \(2\) newform of level \(10\). This is an obvious contradiction.

Next, assume that $\alpha$ and $\beta$ are odd, and $n\geq7$. We consider the elliptic curve \(E=E_{n,n,3}^{A,B,C}(a,b,c)\) with
\[
A=1,\ B=-3^{\alpha_0}7^{\gamma_0},\ C=(-1)^{\delta}5^{\beta_0}, a=z,\ b=3^{\alpha_1}7^{\gamma_1}\text{ and } c=5^{\beta_1},
\]
where \(\alpha=n\alpha_1+\alpha_0\), \(\gamma=n\gamma_1+\gamma_0\) with \(0\leq \alpha_0,\gamma_0\leq n-1\) and \(\beta=3\beta_1+\beta_0\) with \(0\leq \beta_0\leq2\). Then~\(E\sim_nf\) where~\(f\) is a weight \(2\) newform of level, say, \(N_0\) which we may compute using Proposition~\ref{lem:nn3_LL-1}. If \(\alpha\geq3\), we find that \(N_0\mid 3\cdot5^2\cdot7\) and hence reach a contradiction from consideration of  Table~\ref{table:levels}. Assume therefore  that \(\alpha=1\) and that \((A,B,C,a,b,c)\) does not correspond to the solution \(3+5^3=2^7\). Then, necessarily, \(n=7\) and \(N_0=3^4\cdot5^2\cdot7\). In particular, we have \(\gamma_0\not=0\), i.e., \(\gamma\not\equiv0\mod{7}\). We now use local arguments to conclude. Indeed, if \(\gamma\geq2\), then by reducing mod \(49\), we obtain that \(\beta\equiv0\mod{7}\), a contradiction. If, however, \(\gamma=1\), then \(\delta=1\) and we check that there is no value of \(\beta\) (odd) such that \(\beta\not\equiv0\mod{7}\) and 
\[
(3\cdot7-5^{\beta})^{(p-1)/7}\equiv0\text{ or }1\mod{p}
\]
holds for \(p=29\) and \(71\) simultaneously.

This gives the desired contradiction and hence proves that the only solution to~(\ref{eq14}) for \(n\geq 7\) prime corresponds to~\(3+5^3=2^7\).

\subsubsection{The equation $5^{\beta}7^{\gamma}+(-1)^{\delta}3^{\alpha}=z^n$}
By reducing modulo \(8\), we see that \(\alpha\) and \(\beta\) have the same parity and that \(\alpha\) is odd if and only if we have \(\delta\equiv\gamma\mod{2}\). From the preceding two subsections, we may suppose that $\beta \gamma \neq 0$. If $n=7$ and $\gamma \geq 2$, then considering the equation modulo $7^2$, we conclude that \(\alpha\equiv0\mod{7}\), a contradiction. If, however, $n=7$ and $\gamma=1$, we can easily check that there is no solution modulo~\(2^3\cdot7^2\cdot29\cdot43\cdot127\cdot379\). We may thus suppose that $n \geq 11$. Further, if $\alpha,\beta$ and $\gamma$ are all even (so that $\delta=1$),  we can consider the elliptic curve
\[
E=E_{(3),n,n,2}^{3^{\alpha_0},1,1}(3^{\alpha_1},z,(-1)^{\gamma/2}5^{\beta/2}7^{\gamma/2})
\]
where \(\alpha=n\alpha_1+\alpha_0\) with \(0<\alpha_0\leq n-1\). By Proposition~\ref{nn2}, it arises at level \(6\),  a contradiction. We may thus suppose that at least one of $\alpha,\beta$ or $\gamma$ is odd.

 If $\alpha\geq3$, then we consider 
\[
E=E_{n,n,3}^{1,(-1)^{\delta+1}3^{\alpha_0},5^{\beta_0}7^{\gamma_0}}(z,3^{\alpha_1},5^{\beta_1}7^{\gamma_1})
\] 
where
\[
\left\{\begin{array}{ll}
\alpha=n\alpha_1+\alpha_0, & 0<\alpha_0\leq n-1 \\
\beta=3\beta_1+\beta_0, & 0\leq\beta_0\leq 2 \\
\gamma=3\gamma_1+\gamma_0, & 0\leq\gamma_0\leq 2. \\
\end{array}
\right..
\]
By Proposition~\ref{lem:nn3_LL-1}, we have \(E\sim_n f\) where $f$ has level \(N_0\mid 3\cdot5^2\cdot7^2\); from  Table~\ref{table:levels}, we necessarily have \(n=11\).

We now use an \((n,n,n)\) Frey-Hellegouarch curve argument to treat this remaining case. Consider either the elliptic curve
\[
E=E_{11,11,11}^{-5^{\beta_0}7^{\gamma_0},1,(-1)^{\delta}3^{\alpha_0}}(5^{\beta_1}7^{\gamma_1},z,3^{\alpha_1})
\] 
or
\[
E=E_{11,11,11}^{5^{\beta_0}7^{\gamma_0},-1,(-1)^{\delta+1}3^{\alpha_0}}(5^{\beta_1}7^{\gamma_1},z,3^{\alpha_1})
\]
where 
\[
\left\{\begin{array}{ll}
\alpha=11\alpha_1+\alpha_0 & 0<\alpha_0\leq 10, \\
\beta=11\beta_1+\beta_0 & 0\leq\beta_0\leq 10 \\
\gamma=11\gamma_1+\gamma_0 & 0\leq\gamma_0\leq 10, \\
\end{array}
\right.
\] 
if \(\gamma\) is even or odd, respectively. Then, by Proposition~\ref{nnn}, \(E\sim_n f\) where $f$ has level 
\(N_0\in\{30,42,210\}\). It follows that the newform \(f\) corresponds to an elliptic curve $F/\mathbb{Q}$ for which there are unique isogeny classes at level \(N_0\in\{30,42\}\) and five isogeny classes at level \(210\). For \(q\in\{23,67,89,199\}\), we compute~\(a_q(F)\) to show that \(E\) has good reduction at~\(q\)  (i.e., \(q\nmid z\)) unless, perhaps, if either \(q=89\) and \(f\) corresponds to the isogeny class \verb"210c" (in Cremona's notation), or \(q=199\) and \(f\) corresponds to the isogeny class \verb"210d". We finally sieve over \(\alpha,\beta\) and \(\gamma\) (not all even) to show that we do not have
\[
5^{\beta}7^{\gamma}+(-1)^{\delta}3^{\alpha}\in\left((\mathbb{Z}/q\mathbb{Z})^{\times}\right)^{11}\quad\text{and}\quad a_q(E)\equiv a_q(F)\mod{11},
\] 
for \(q=23,67\) and \(89\) simultaneously if the isogeny class of~\(F\) is not \verb"210c" and for \(q=23,67\) and \(199\) simultaneously otherwise.

Assume now $\alpha\in\{1,2\}$. We basically follow the same strategy we used for the case \(\alpha\geq3\) applying first an \((n,n,3)\) and then an \((n,n,n)\) argument, though the details are somewhat simpler. Indeed, we first consider 
\[
E=E_{n,n,3}^{5^{\beta_0}7^{\gamma_0},-1,(-1)^{\delta+1}3}(5^{\beta_1}7^{\gamma_1},z,1)
\] 
or
\[
E=E_{n,n,3}^{-5^{\beta_0}7^{\gamma_0},1,(-1)^{\delta}3^{2}}(5^{\beta_1}7^{\gamma_1},z,1)
\]
where
\[
\left\{\begin{array}{ll}
\beta=n\beta_1+\beta_0, & 0\leq\beta_0\leq n-1 \\
\gamma=n\gamma_1+\gamma_0, & 0\leq\gamma_0\leq n-1 \\
\end{array}
\right.
\] 
if \(\alpha=1\) or \(\alpha=2\) respectively. Then, by Proposition~\ref{lem:nn3_LL-1}, \(E\sim_n f\) where $f$ has level  $3^5 \cdot N_1$, where $N_1 \mid 5\cdot7$. From appeal to Table~\ref{table:levels}, we conclude that \(n=17\) (and \(\beta,\gamma\not\equiv0\mod{17}\)).

We now use an \((n,n,n)\) Frey-Hellegouarch curve argument to deal with the case \(n=17\). As before, we consider the elliptic curve
\[
E=E_{17,17,17}^{-5^{\beta_0}7^{\gamma_0},1,(-1)^{\delta}3^{\alpha}}(5^{\beta_1}7^{\gamma_1},z,1)
\] 
or
\[
E=E_{17,17,17}^{5^{\beta_0}7^{\gamma_0},-1,(-1)^{\delta+1}3^{\alpha}}(5^{\beta_1}7^{\gamma_1},z,1)
\]
with \(\beta_0,\beta_1,\gamma_0,\gamma_1\) as above, if \(\gamma\) is even or odd, respectively. Then \(E\) arises mod \(17\) from an elliptic curve \(F\) of conductor \(N_0\in\{30,42,210\}\) and for \(q\in\{103,137\}\), we check, by computing~\(a_q(E)\mod{17}\) that \(E\) has good reduction at~\(q\) (i.e., that \(q\nmid z\)). For \(\alpha\in\{1,2\}\), we finally sieve over \(\beta\) and \(\gamma\) to show that 
\[
5^{\beta}7^{\gamma}+(-1)^{\delta}3^{\alpha}\in\left((\mathbb{Z}/q\mathbb{Z})^{\times}\right)^{17}\quad\text{and}\quad a_q(E)\equiv a_q(F)\mod{17},
\] 
does not hold for \(q=103\) and \(137\) simultaneously. 
This finishes the proof that equation~(\ref{eq15}) has no solution for \(n\geq7\).

\subsubsection{The equation $3^{\alpha}+(-1)^{\delta}5^{\beta}=7^{\gamma}z^n$}
Considering the equation \(8\), we conclude that \(\alpha\) and \(\beta\) have the same parity and that \(\alpha\) is odd if and only if we have \(\delta=0\). If  $\alpha$ and $\beta$ are odd, then $\delta=0$ and, since  \(\gamma>0\), we have \(3^{\alpha}+5^{\beta}\equiv0\mod{7}\)  and hence a contradiction. Next, assume that  $\alpha$ and $\beta$ are even and $n\geq7$. The equation to treat is now 
\[
3^{\alpha}-5^{\beta}=7^{\gamma}z^n.
\] 
We thus consider the elliptic curve
\[
E=E_{n,n,3}^{-7^{\gamma},3^{\alpha_0},5^{\beta_0}}(z,3^{\alpha_1},5^{\beta_1})
\] 
where 
\[
\left\{\begin{array}{ll}
\alpha=n\beta_1+\alpha_0, & 0\leq\beta_0\leq n-1 \\
\beta=3\beta_1+\beta_0, & 0\leq\gamma_0\leq 2 \\
\end{array}
\right..
\] 
By Proposition~\ref{lem:nn3_LL-1}, we have \(E\sim_n f\) where \(f\) has level \(N_0=3^k\cdot5^{\delta_5}\cdot7\) with \(\delta_5\in\{0,2\}\) and
\[
k=\left\{\begin{array}{ll}
3 & \text{if \(\alpha=2\)}\\
0 & \text{if \(\alpha\equiv3\mod{n}\)}\\
1 & \text{if \(\alpha\geq3\) and \(\alpha\not\equiv3\mod{n}\)}\\
\end{array}
\right..
\] 
Using Table~\ref{table:levels}, it then follows that $\alpha=2$ and $n \in \{ 7, 11 \}$.
We first consider the \((n,n,3)\) Frey-Hellegouarch curve 
\[
E=E_{n,n,3}^{7^{\gamma},5^{\beta_0},3^2}(z,5^{\beta_1},1),
\] 
where  \(\beta=n\beta_1+\beta_0\) with \(0\leq\beta_0\leq n-1\). By Proposition~\ref{lem:nn3_LL-1}, the curve \(E\) arises modulo~\(n\) from a newform \(f\) of level  $N_0=3^5 \cdot N_1$, where $N_1 \mid 5\cdot7$. Hence, using Table~\ref{table:levels}, we deduce that \(n\in\{7,17\}\) and, in particular, that $n \neq 11$. 

If, however, \(n=7\), then we use local arguments to get a contradiction. If \(\gamma\geq2\), we check that there is no solution modulo~\(2^3\cdot7^2\cdot43\). Similarly, if \(\gamma=1\), then there is no solution modulo~\(2^3\cdot7^2\cdot43\cdot127\).
This shows that equation~(\ref{eq16}) has no solution for~\(n\geq7\) prime.

\subsubsection{The equation $3^{\alpha}+(-1)^{\delta}7^{\gamma}=5^{\beta}z^n$}
By reducing mod \(2^4\cdot5\), we conclude that \(\alpha\equiv0\mod{4}\), \(\gamma\equiv0\mod{4}\) and \(\delta=1\). Write \(\alpha=2\alpha_0\) and \(\gamma=2\gamma_0\). Then, there exist nonzero integers \(z_1,z_2\) with \(z_2\) odd such that either
\[
\left\{
\begin{array}{l}
3^{\alpha_0}-7^{\gamma_0}=2^{n-1}5^{\beta}z_1^{n} \\
3^{\alpha_0}+7^{\gamma_0}=2z_2^{n} \\
\end{array}
\right.
\quad\text{or}\quad
\left\{
\begin{array}{l}
3^{\alpha_0}-7^{\gamma_0}=2^{n-1}z_1^{n} \\
3^{\alpha_0}+7^{\gamma_0}=2\cdot5^{\beta}z_2^{n} \\
\end{array}
\right..
\]
Adding these equations and recalling that \(\alpha_0\) is even yields either 
\[
\left(3^{\alpha_0/2}\right)^2=2^{n-2}5^{\beta}z_1^n+z_2^n\quad\text{or}\quad \left(3^{\alpha_0/2}\right)^2=2^{n-2}z_1^n+5^{\beta}z_2^n.
\] 
If \(n\geq11\), we consider the \((n,n,2)\) Frey-Hellegouarch curve 
\[
E=E_{(3),n,n,2}^{1,2^{n-2}5^{\beta},1}(z_2,z_1,(-3)^{\alpha_0/2})\quad\text{or}\quad E=E_{(3),n,n,2}^{5^{\beta},2^{n-2},1}(z_2,z_1,(-3)^{\alpha_0/2})
\] 
according to whether we are in the first or second case above, respectively. Then, by Proposition~\ref{nn2}, the curve \(E\) arises at level \(10\), an immediate contradiction. We may therefore assume that \(n=7\) and sieve over $\alpha,\beta$ and $\gamma$ to show that there is no solution modulo \(2^4\cdot7^2\cdot29\cdot43\cdot71\cdot113\cdot127\cdot211\cdot337\cdot421\). This shows that equation~(\ref{eq17}) has no solution for~\(n\geq7\) prime.

\subsubsection{The equation $5^{\beta}+(-1)^{\delta}7^{\gamma}=3^{\alpha}z^n$}
By reducing mod \(2^4\cdot3\), we find that \(\beta\equiv0\mod{4}\), \(\gamma\equiv0\mod{2}\) and \(\delta=1\). We then consider the \((n,n,2)\) Frey-Hellegouarch curve 
\[
E=E_{(3),n,n,2}^{5^{\beta_0},-3^{\alpha},1}(5^{\beta_1},z,(-7)^{\gamma/2}),
\] 
where \(\beta=n\beta_1+\beta_0\) with \(0\leq \beta_0\leq n-1\). It has good reduction at \(7\) and satisfies \(a_7(E)=0\). On the other hand, the curve \(E\) arises from the (unique)  newform \(f\) at level \(30\), which  satisfies \(c_7(f)=-4\). This gives us the desired contradiction and hence proves that equation~(\ref{eq18}) has no solution for~\(n\geq7\) prime.

\bibliographystyle{amsplain}

\end{document}